\documentclass[envcountsect,oribibl,orivec]{llncs}
\usepackage{amssymb}
\usepackage{amsmath}
\usepackage{bussproofs}

\EnableBpAbbreviations

\newcommand{\m}{{\sf J(\mu)}}

\newcommand{\s}{{\sf S4}}
\newcommand{\LP}{{\sf LP}}
\newcommand{\GL}{{\sf GL}}
\newcommand{\PA}{{\sf PA}}
\newcommand{\M}{{\mathcal M}}
\newcommand{\E}{{\mathcal E}}
\newcommand{\I}{{\mathcal I}}
\newcommand{\V}{{\mathcal V}}
\newcommand{\CS}{\mathcal{CS}}
\newcommand{\JL}{{\sf JL}}
\newcommand{\QLP}{{\sf QLP}}
\newcommand{\gn}[1]{\left\ulcorner #1\right \urcorner\,\!}

 \def\f{\frac}

 \def\di{\displaystyle}
 \def\r{\rightarrow}
 \def\b{\Box}

 \def\vd{\vdash}

\begin{document}
\title{A Note on Fixed Points in Justification Logics and the Surprise Test Paradox}

\author{Meghdad Ghari}
\institute{School of Mathematics,
Institute for Research in Fundamental Sciences (IPM), \\ P.O.Box: 19395-5746, Tehran, Iran \\ \email{ghari@ipm.ir}
}

\maketitle
\begin{abstract}
In this note we study the effect of adding fixed points to justification logics. We introduce two extensions of justification logics: extensions by fixed point (or diagonal) operators, and extensions by least fixed points. The former is a justification version of Smory\`{n}ski's Diagonalization Operator Logic, and the latter is a justification version of Kozen's modal $\mu$-calculus. We also introduce fixed point extensions of Fitting's quantified logic of proofs, and formalize the Knower Paradox and the Surprise Test Paradox in these extensions. By interpreting a surprise statement as a statement for which there is no justification, we give a solution to the self-reference version of the  Surprise Test Paradox in quantified logic of proofs. We also give formalizations of the Surprise Test Paradox in timed modal epistemic logics, and in G\"{o}del-L\"{o}b provability logic.\\

\end{abstract}

\section{Introduction}
Justification logics provide a framework for reasoning about mathematical proofs and epistemic justifications. Justification logics evolved from a logic called \textit{Logic of Proofs} $\LP$, introduced by Sergei Artemov in \cite{A1995,A2001}, which try to give an arithmetic semantics for modal logic
\s~and intuitionistic logic, and formalize the Brouwer-Heyting-Kolmogrov semantics of intuitionistic logic. Justification logics are extensions of classical logics by justification assertions $t:F$, which is read as ``$t$ is a justification for $F$." Some of the justification logics enjoy the arithmetical completeness theorem, with the provability semantics of $t:F$ as ``$t$ is a proof of $F$ in Peano arithmetic $\PA$."

Justification logics could be also considered as logics of knowledge, and are contributed to the study of \textit{Justified True Belief} vs. \textit{Knowledge} problem. In this respect \LP~can be also viewed as a refinement of
the epistemic logic \s, in which knowability operator $\Box A$
($A$ is known) is replaced by explicit knowledge operators $t:A$
(``$F$ is known for reason $t$"). The exact correspondence between
\LP~and \s~is given by the \textit{Realization Theorem}: all
occurrences of $\b$ in a theorem of \s~can be replaced by suitable
terms to produce a theorem of \LP, and vice versa. Regarding this
theorem, \LP~is called the justification counterpart
of \s. The justification counterpart of other modal logics were also developed
(see e.g. \cite{Brezhnev2000,Ghari2012-Thesis,GoetschiKuznets2012}).

Since some of the justification logics enjoy the arithmetical completeness theorem, it is natural to ask if the ability of constructing self-reference statements in $\PA$, by means of the G\"{o}del's Fixed Point (or Diagonal) Lemma, can be simulated in justification logics.  In the context of language, a statement is self-reference if it refers to itself or its referent. The Fixed Point Lemma in $\PA$ enables us to construct sentences which behaves like the self-reference sentences. Such a self-reference statements have been used to show important results in $\PA$ and arise significant philosophical issues, e.g. G\"{o}del's incompleteness theorems (by constructing a sentence which state its own unprovability), Tarski's undefinability of truth (by constructing a sentence which state its own falsity, the Liar Paradox), Kaplan-Montague's Knower Paradox (by constructing a sentence which state its own unknowability).

In the framework of modal logics, the G\"{o}del-L\"{o}b provability logic $\GL$ is one of the well-known modal logics which is arithmetically complete. The fixed point lemma is formulated in $\GL$ by the De Jongh-Sambin Fixed Point Theorem. Most of the proofs of the De Jongh-Sambin Fixed Point Theorem employs the property of Substitution of Equivalents
\[\frac{A \leftrightarrow B}{C[A] \leftrightarrow C[B]}SE\]
for a context $C[~]$. The proof of SE in modal logics requires the Regularity rule
\[\frac{A \leftrightarrow B}{\b A \leftrightarrow \b B}Reg\]
Obviously the justification version of the Regularity rule does not hold in justification logics
\[\frac{A \leftrightarrow B}{t:A \leftrightarrow  t:B}JReg\]
In other words, two equivalent statements have not necessarily the same justifications (see \cite{Fitting2009b} for a version of SE in the logic of proofs). Thus, instead of proving a fixed point theorem in the framework of justification logics, we extend the language and axioms of justification logics by fixed point formulas and fixed point axioms respectively. We consider two extensions of justification logics: extensions by fixed point (or diagonal) operators, and extensions by least  (and greatest) fixed points. The former is a justification version of Smory\`{n}ski's Diagonalization Operator Logic \cite{Smorynski1985}, and the latter is a justification version of Kozen's modal $\mu$-calculus \cite{Kozen1983}. In this paper, we do not introduce any semantics for these extensions. However, the consistency of  some of these extensions are shown by translating them into their counterpart modal logics.

We also introduce fixed point extensions of Fitting's quantified logic of proofs \cite{Fitting2008}, and formalize the Knower Paradox and the Surprise Test Paradox in these extensions. By interpreting a surprise statement as a statement for which there is no justification, we give a solution to the one-day case self-reference version of the  Surprise Test Paradox in quantified logic of proofs. To this end, we give a simple semantics (single-world Kripke models) for quantified logic of proofs. We also show that the one-day case non-self-reference version of the  paradox is an epistemic blindspot for students (cf. \cite{Sorensen1984}).

\section{Fixed points in arithmetic}\label{sec:Fixed points in arithmetic}
In this section we recall some well known consequences of the \textit{Fixed Point Lemma} (or \textit{Diagonal Lemma}) in extensions of Peano Arithmetic $\PA$.\footnote{All the results also hold for extensions of Robinson arithmetic ${\sf Q}$.} In this paper we do not distinguish between the number $n$ and its numeral $\bar{n}$. The G\"{o}del number of formula $A$ is denoted by $\ulcorner A\urcorner$. The following (generalized) Fixed Point Lemma is taken from \cite{Boolos1993}.

\begin{lemma}[Fixed Point Lemma]\label{lem:Fixed Point Lemma PA}
Let $T$ be a theory extending $\PA$. For every
formula $\varphi(x,y_1,\ldots,y_n)$ there exists a formula
$D(y_1,\ldots,y_n)$ such that
\[T\vdash D(y_1,\ldots,y_n) \leftrightarrow \varphi(\ulcorner D(y_1,\ldots,y_n)\urcorner,y_1,\ldots,y_n).\]
\end{lemma}

This lemma enables us to formalize self-references sentences in $\PA$. G\"{o}del uses this lemma\footnote{As pointed out by Mendelson \cite{Mendelson1997}, the Fixed Point Lemma was implicit in the paper of G\"{o}del \cite{Godel1931}, and it seems first explicitly mentioned by Carnap \cite{Carnap1934}.} to construct a sentence in $\PA$ which states ``I am not provable in $\PA$."

\begin{theorem}[G\"{o}del's Incompleteness Theorem]\label{thm:Incompleteness Theorem}
Let $T$ be a recursively axiomatized complete theory extending $\PA$. Then $T$ is inconsistent.
\end{theorem}
\begin{proof}
G\"{o}del constructed a proof predicate $Proof(x,y)$ in $T$ such that for every formula $\varphi$:
\begin{equation}\label{eq:Proof}
T\vdash \varphi ~\textrm{if{f}}~ T\vdash Proof(n,\gn{\varphi}), \textrm{for some}~n\in\mathbb{N}.
\end{equation}
Let $Prov(x)$ stand for $(\exists y) Proof(y,x)$. It is easy to show that
\begin{equation}\label{eq:Prov}
T\vdash\varphi ~\textrm{if{f}}~ T\vdash Prov(\gn{\varphi}).
\end{equation}
Now by the Fixed Point Lemma there is a sentence $\mathcal{G}$ (G\"{o}del's sentence for T) such that
\begin{equation}\label{eq:Godel sentence}
T\vdash \mathcal{G} \leftrightarrow \neg Prov(\gn{\mathcal{G}}).
\end{equation}
Since $T$ is assumed to be complete, then either $T\vdash \mathcal{G}$ or $T\vdash \neg\mathcal{G}$. We will show that in either case $T$ is inconsistent.

If $T\vdash \mathcal{G}$, then by (\ref{eq:Prov}) we have $T\vdash Prov(\gn{\mathcal{G}})$. Next by (\ref{eq:Godel sentence}) we obtain $T\vdash \neg \mathcal{G}$. Thus, $T$ proves both $\mathcal{G}$ and $\neg\mathcal{G}$.

If $T\vdash \neg\mathcal{G}$, then by (\ref{eq:Godel sentence}) we have $T\vdash Prov(\gn{\mathcal{G}})$. Next by (\ref{eq:Prov}) we obtain $T\vdash \mathcal{G}$. Thus, $T$ proves both $\mathcal{G}$ and $\neg\mathcal{G}$.\qed
\end{proof}

\begin{theorem}[Tarski's Undefinability of Truth]\label{thm:Tarski's Undefinability of Truth}
Let $T$ be a theory extending $\PA$, and $Tr(x)$
be a truth predicate, i.e a predicate with one free variable $x$ such that for
every sentence $A$
\begin{description}
    \item[Tr.] $T\vdash A \leftrightarrow Tr(\ulcorner A\urcorner)$.
\end{description}
Then T is inconsistent.
\end{theorem}

\begin{proof}
\begin{enumerate}
    \item $ D \leftrightarrow \neg Tr(\gn{D})$, by the Fixed Point Lemma
    \item $ D \leftrightarrow  Tr(\gn{D})$, by  Tr
    \item $ Tr(\gn{D}) \leftrightarrow\neg  Tr(\gn{D})$,
    contradiction.\qed
\end{enumerate}
\end{proof}
In fact, the sentence $D$ in $ D \leftrightarrow \neg Tr(\gn{D})$  correspond to the Liar Sentence:
\begin{center}
``This statement is false"
\end{center}
and the argument given in the proof of Tarski's theorem, which expresses that the Liar Sentence is true if and only if it is not, is known as the \textit{Liar Paradox}. The scheme Tr is called the \textit{Tarski biconditional} or \textit{T-scheme}. Note that, by (\ref{eq:Prov}), the provability predicate $Prov(x)$ satisfies the T-scheme.

In the following we consider the \textit{Surprise Test Paradox}. This paradox first published by O'Connor \cite{Connor} with the name ``Class A blackout." In the following we give the more common formulation of the paradox, the Surprise Test (or Examination) Paradox, which is given by Weiss \cite{Weiss} (under a different name ``the prediction paradox"). For a survey of the paradox see \cite{Chow,Sorensen2014}.

The two-day case of this paradox is as follows:
\begin{quote}
``A teacher announces that there will be exactly one surprise test on Wednesday or Friday next week. A student objects that this is impossible. If the test is given on Friday, then on Thursday I would be able to predict that the test is on Friday. It would not be a surprise. The test could not be given on Wednesday too. Because on Tuesday I would know that the test will not be on Friday (as shown in the previous reasoning) and therefore I could foresee that the test will be on Wednesday. Again a test on Wednesday would not be a surprise. Therefore, it is impossible for there to be a surprise test."
\end{quote}
The one-day case of the paradox is as follows:
\begin{quote}
``You will have a test tomorrow that will take you by surprise, i.e. you can't know it beforehand"
\end{quote}
As it is clear from the above formulations of the paradox, ``surprise test" is defined in terms of what can be known. Specifically, a test is a surprise for a student if and only if the student cannot know beforehand which day the test will occur. The precise formulation of (non-self-reference) $n$-day case of the paradox requires temporalized knowledge operators in the language. A formalization of the student's argument in the two-day case in timed modal epistemic logics is given in Appendix \ref{Appendix:Surprise test paradox in tMEL}.

Some authors interpret the surprise (or knowledge) in this paradox in terms of deducibility. Specifically, a test is a surprise for a student if and only if the student cannot deduce logically beforehand the date of the test. This interpretation was first proposed by Shaw \cite{Shaw}. Regarding deducibility in \PA, Fitch \cite{Fitch1964} resolved the paradox by reinterpreting the surprise so that ``what is intended in practice is not that the surprise event will be a surprise \textit{whenever} it occurs, but only when it occurs on some day \textit{other than the last}". He formulated the announcement with this interpretation and claim that it is ``apparently self-consistent" in \PA. Nevertheless, Kripke \cite{Kripke2011} shows that the statement of Fitch's resolution is actually refutable in \PA. Fitch also show the relation between a version of the Surprise Test Paradox and G\"{o}del's first incompleteness theorem. Kritchman and Raz \cite{KritchmanRaz} also show the relationship between the paradox and G\"{o}del's second incompleteness theorem. They conclude that ``if the students believe in the consistency of $T+S$, the exam cannot be held on Friday [i.e. the last day], ... However, the exam can be held on any other day of the week because the students cannot prove the consistency of $T+S$" (where $T$ can be taken $\PA$, and $S$ is the statement of teacher's announcement). The formalization of these results in G\"{o}del-L\"{o}b provability logic are presented in Appendix \ref{Appendix:Surprise test paradox in GL}.

The self-reference version of the paradox (adopted from \cite{Egre2005,KaplanMontague1960}) is as follows:
\begin{quote}
 ``Unless you know this statement to be false, you will have a test tomorrow, but you can't know from this statement that you will have a test tomorrow."
\end{quote}
The above version of the paradox is called the \textit{Examiner Paradox} in \cite{Egre2005}. \`{E}gr\`{e} defined a \textit{knowledge predicate} as a predicate satisfying the principle of knowledge veracity: $K(\gn{A})\r A$, for every sentence $A$. Now, using a knowledge predicate, the Examiner Paradox is formalized as follows:
\begin{equation}\label{eq:Self-reference Examiner Paradox}
D \leftrightarrow ( K(\gn{\neg  D})\vee( E\wedge \neg K(\gn{D\r E}))),
\end{equation}
  where $ E$ denotes the sentence ``you will have a test tomorrow." Using (\ref{eq:Self-reference Examiner Paradox}), \`{E}gr\`{e} proved the following.

  \begin{theorem}[\cite{Egre2005}]
Let $T$ be a theory extending $\PA$, with $I(x,y)$ a formula expressing derivability between formulas of $T$, and $K(x)$ a unary predicate such that for every sentence $A$
\begin{description}
  \item[T.] $K(\gn{A})\r A$,
  \item[U.] $K(\ulcorner K(\gn{A})\r A\urcorner)$,
  \item[I.] $K(\gn{A})\wedge I(\gn{A},\gn{B})\r K(\gn{B})$,
  \item[R.] $K(\ulcorner T\wedge U\wedge I\urcorner)$.
\end{description}
Then $T$ is inconsistent.
\end{theorem}

Assumptions U, I, and  R could be replaced with a stronger assumption, a rule similar to the modal necessitation rule, as follows.

  \begin{theorem}\label{thm:Examiner paradox in PA}
Let $T$ be a theory extending $\PA$, and $K(x)$ a unary predicate such that for
every sentence $A$
\begin{description}
  \item[T.] $K(\gn{A})\r A$,
  \item[Nec.] If $T\vd A$, then $T\vd K(\gn{A})$,
\end{description}
Then $T$ is inconsistent.
\end{theorem}
\begin{proof}
\begin{enumerate}
    \item $  D \leftrightarrow (K(\gn{\neg D})\vee( E\wedge \neg K(\gn{ D\r E})))$, by the Fixed Point Lemma
    \item $ K(\gn{\neg D})\r \neg D$, by T
    \item $ D\r\neg K(\gn{\neg D})$, by 2 and propositional reasoning
    \item $  D \rightarrow ( E\wedge \neg K(\gn{ D\r E}))$, by 1, 3 and propositional reasoning
    \item $ D\r E$, by 4 and propositional reasoning
    \item  $ D\r\neg K(\gn{ D\r E})$, by 4 and propositional reasoning
    \item $K(\gn{ D\r E})$, by 5 and Nec
    \item $\neg D$, by 6, 7 and propositional reasoning
    \item $ K(\gn{\neg D})$, by 8 and Nec
    \item $ D$, by 1, 9 and propositional reasoning
    \item $\bot$, by 8 and 10.\qed
\end{enumerate}
\end{proof}

Now we consider the zero-day case Surprise Test Paradox, which is known as the \textit{Knower Paradox}. It is formulated as $D \leftrightarrow  K(\gn{\neg  D})$ that states:
\begin{center}
``This statement is known to be false"
\end{center}
or as $D \leftrightarrow  \neg K(\gn{D})$ that states:
\begin{center}
``Nobody knows this statement to be true"
\end{center}
The original formulation of the Knower Paradox presented by Kaplan and Montague in \cite{KaplanMontague1960}, and basically is the epistemological counterpart of the Liar Paradox.

\begin{theorem} [The Knower Paradox, \cite{KaplanMontague1960}]
Let $T$ be a theory extending $\PA$, with $I(x,y)$ a formula expressing derivability between formulas of $T$, and $K(x)$ a unary predicate such that for every sentence $A$ and $B$
\begin{description}
  \item[T.] $K(\gn{A})\r A$,
  \item[U.] $K(\gn{K(\gn{A})\r A})$,
  \item[I.] $K(\gn{A})\wedge I(\gn{A},\gn{B})\r K(\gn{B})$.
\end{description}
Then $T$ is inconsistent.
\end{theorem}

Similar to the Tarski's Undefinability of Truth (Theorem \ref{thm:Tarski's Undefinability of Truth}), the Knower Paradox can be seen as the \textit{Arithmetic Undefinability of Knowledge}. The following variant of  the Knower Paradox is given  by  Montague in \cite{Montague1963}.

\begin{theorem} [\cite{Montague1963}]\label{thm:Knower, simple case}
Let $T$ be a theory extending $\PA$, and $K(x)$ a unary predicate such that for every sentence $A$
\begin{description}
 \item[T.] $K(\gn{A})\r A$,
  \item[Nec.] If $T\vd A$, then $T\vd K(\gn{A})$.
\end{description}
Then $T$ is inconsistent.
\end{theorem}
\begin{proof}
\begin{enumerate}
    \item $ D\leftrightarrow\neg K(\gn{D})$, by the Fixed Point Lemma
    \item $ K(\gn{D})\r D$, by T
    \item $ K(\gn{D})\r \neg K(\gn{D})$, by 1, 2 and propositional reasoning
    \item $ \neg K(\gn{D})$, by 3 and propositional reasoning
    \item $ D$, by 1, 4 and propositional reasoning
    \item $ K(\gn{D})$, by 5 and Nec
    \item $\bot$, by 4, 6 and propositional reasoning.\qed
\end{enumerate}
\end{proof}
Note that Theorem \ref{thm:Knower, simple case} is a generalization of Theorem \ref{thm:Tarski's Undefinability of Truth}, since every truth predicate is a knowledge predicate satisfying also the rule Nec.

Finally, we give the following version of the \textit{Believer Paradox} from \cite{Egre2005}.
 \begin{theorem}\label{thm:Believer, simple case}
Let $T$ be a theory extending $\PA$, and $B(x)$ a unary predicate such that for every sentence $F$ and $G$
\begin{description}
 \item[K.] $B(\gn{F\r G})\r (B(\gn{F}) \r B(\gn{G}))$,
 \item[4.] $B(\gn{F}) \r B(\gn{B(\gn{F})})$,
 \item[D.] $B(\gn{\neg F}) \r \neg B(\gn{F})$,
  \item[Nec.] If $T\vd F$, then $T\vd B(\gn{F})$.
\end{description}
Then $T$ is inconsistent.
\end{theorem}
\section{Fixed points in modal logics}\label{sec:Fixed points in modal logics}
In this section, we recall two kind of extensions of modal logics with fixed points: fixed point (or diagonal) extensions of modal logics where first introduced by Smory\`{n}ski in \cite{Smorynski1985}, and modal $\mu$-calculus (modal fixed point logics) where first introduced by Kozen in \cite{Kozen1983}. We first recall definitions of normal modal logics.

Modal formulas are constructed by the following grammar:
\[ A::= p~|~\bot~|~\neg A~|~A\wedge A~|~A\vee A~|~A\rightarrow A~|~\b A,\]
where $p$ is a propositional variable,  $\bot$ is the propositional
constant for falsity. As usual, $\Diamond A$ is defined as $\neg \b \neg A$.\footnote{We adopt the following decreasing order of strength of the connectives: $\neg,\Box,\wedge,\vee,\r$.}

The basic modal logic {\sf K} has the following axiom schemes and rules:
\begin{description}
\item[Taut.] All propositional tautologies,
  \item[K.] $\b(A\r B)\r(\b A\r \b B)$,
\end{description}
The rules of inference are \textit{Modus Ponens} and \textit{Necessitation rule}:
\begin{description}
\item[MP.]  from $\vdash A$ and $\vdash A\r B$, infer $\vdash B$.
\item[Nec.]  from $\vdash A$, infer $\vdash \b A$.
\end{description}
Other modal logics are obtained by adding the following axiom schemes to {\sf K} in various combinations:
\begin{description}
\item[T.] $\b A\r A$.
\item[D.] $\b A\r\Diamond A$
\item[4.] $\b A\r\b\b A$.
\item[B.] $\neg A\r\b\neg\b A$.
\item[5.] $\neg \b A\r\b\neg\b A$.
\end{description}
In this paper we consider the following 15 normal modal logics: {\sf K}, {\sf T}, {\sf D}, {\sf K4}, {\sf KB}, {\sf K5}, {\sf KB5}, {\sf K45}, {\sf D5}, {\sf DB}, {\sf D4}, {\sf D45}, {\sf TB}, {\sf S4}, {\sf S5}. The name of each modal logic indicates the list of its axioms, except {\sf S4} and {\sf S5} which can be named {\sf KT4} and {\sf KT45}, respectively. The axiom D is equivalent (over {\sf K}) to $\b \bot\r\bot$.

The G\"{o}del-L\"{o}b provability logic $\GL$ has a central role here. $\GL$ is obtained from ${\sf K4}$ (or ${\sf K}$) by adding the L\"{o}b axiom scheme:
$$\b(\b A\r A)\r\b A$$
It is known that $\GL$ is complete with respect to provability interpretation of $\b$ in $\PA$. More precisely, define the arithmetical interpretation $*$ as a mapping from propositional variables to sentences of $\PA$, such that $\bot^* =\bot$, $*$ commutes with propositional connectives, and
 $(\b A)^* = Prov(\gn{A^*}).$
The modal logic {\sf K4} is sound with respect to arithmetical interpretations (see \cite{Boolos1993,Smorynski1985}). Moreover, Solovay \cite{Solovay1976} shows that $\GL$ is sound and complete.

\begin{theorem}[Solovay Arithmetical Completeness, \cite{Solovay1976}]\label{thm:Solovay Arithmetical Completeness GL}
$\GL\vdash F$ if{f} $\PA\vdash F^*$,  for every arithmetical interpretation $*$.
\end{theorem}

\subsection{Modal logics with fixed point operators}\label{sec:ML(FP)}
 Suppose {\sf ML} is a propositional modal logic defined over a language ${\cal L}$. We write $A(p,q_1,\ldots,q_n)$ to denote that $p,q_1,\ldots,q_n$ are all the propositional variables occurring in the formula $A$.  An occurrence of a propositional variable $p$ in the formula $A(p,q_1,\ldots,q_n)$ is called \textit{modalized} if $p$ occurs in the scope of a modal operator $\b$ or $\Diamond$. Let ${\cal L}({\sf FP})$ be the extension of ${\cal L}$ by n-ary \textit{fixed point operators} (or \textit{diagonal operators}) $\delta_A(q_1,\ldots,q_n)$ for each ${\cal L}$-formula $A(p,q_1,\ldots,q_n)$ in which $p$ is modalized. The \textit{fixed point extension} (or \textit{diagonal extension}) ${\sf ML(FP})$ of modal logic ${\sf ML}$ in the language ${\cal L}({\sf FP})$ is an extension of ${\sf ML}$ by axiom schemes
\[ \delta_A(B_1,\ldots,B_n) \leftrightarrow A(\delta_A(B_1,\ldots,B_n),B_1,\ldots,B_n),\]
where $B_1,\ldots,B_n$ are ${\cal L}({\sf FP})$-formulas.

Using fixed point extensions of modal logics we can give the analogs of the Knower and the Believer Paradoxes.
\begin{theorem} [The Knower Paradox in the framework of modal logics]\label{Thm Knower Paradox-modal case}
Let {\sf ML} be a propositional modal logic which contain the axiom scheme
\begin{description}
  \item[T.] $\b A\r A$,
  \end{description}
then ${\sf ML(FP)}$ is inconsistent.
\end{theorem}
\begin{proof}
The proof is obtained from the proof of Theorem \ref{thm:Knower, simple case} by replacing the knowledge predicate $K$ by $\b$.\qed
\end{proof}

\begin{theorem}[The Believer Paradox in the framework of modal logics]
Let {\sf ML} be a propositional normal modal logic which contain the axiom schemes
  \begin{description}
  \item[D.] $\b A\r\Diamond A$,
  \item[4.] $\b A\r \b\b A$
  \end{description}
then ${\sf ML(FP)}$ is inconsistent.
\end{theorem}

Thus, for example, the systems ${\sf T(FP)}$, ${\sf S4(FP)}$, and ${\sf D4(FP)}$ are inconsistent. On the other hand, by the De Jongh-Sambin Fixed Point Theorem, ${\sf K4(FP)}$ is a conservative extensions of ${\sf GL}$, and hence ${\sf K4(FP)}$  is consistent.

\begin{theorem}[De Jongh-Sambin Fixed Point Theorem, \cite{Boolos1993,Smorynski1985}]\label{thm:Fixed Point Theorem GL}
For any \GL-formula $A(p,\bar{q})$ in which $p$ is modalized, there exists a unique formula $D(\bar{q})$ such that $$\GL\vd D(\bar{q})\leftrightarrow A(D(\bar{q}),\bar{q}).$$
\end{theorem}

Smory\`{n}ski in \cite{Smorynski1985} showed that ${\sf K4(FP)}$ is a conservative extensions of ${\sf GL}$.\footnote{${\sf K4(FP)}$ is called the \textit{Diagonalization Operator Logic} DOL, in \cite{Smorynski1985}.}
\begin{theorem}[\cite{Smorynski1985}]
Given any \GL-formula $A$, the following are equivalent:
\begin{itemize}
\item $\GL\vdash A$.
\item ${\sf K4(FP)}\vdash A$.
\end{itemize}
\end{theorem}

In fact $\GL$ is a modal logic with built-in fixed point property (the Fixed Point Lemma of arithmetic is modally expressible in $\GL$ by the De Jongh-Sambin Fixed Point Theorem). In $\GL$ the following is provable
\begin{equation}\label{eq:Lob formula fixed point}
\b(\b A\r A) \leftrightarrow \b A.
\end{equation}
Thus $\b A$ is the fixed point of $B(p)=\b(p\r A)$. As it is pointed out in \cite{SambinValentini1982}, the De Jongh-Sambin Fixed Point Theorem shows that ``how the single instance LF [here the fixed point equation (\ref{eq:Lob formula fixed point})] is sufficient to yield the strongest version of diagonalization expressible in the language of modal logic."

There are other modal logics with the fixed point property:
\begin{itemize}
\item G\"{o}del-L\"{o}b-Solovay provability logic ${\sf GLS}$ introduced by Solovay in \cite{Solovay1976}. Axioms of {\sf GLS} are all theorems of $\GL$ and the axiom scheme $\b A\r A$, and its only rule of inference is Modus Ponens. \`{E}gr\`{e} in \cite{Egre2005} claims that {\sf GLS} gives a solution to the Knower Paradox via a hierarchy of rules (Modus Ponens could be applied to all theorems, while Necessitation rule could be applied only to theorems of \GL).
\item Sacchetti in \cite{Sacchetti2001} introduced two families of modal logics with the fixed point property: ${\sf K}+\b(\b^n A\r A)\r\b A$, for $n\geq 1$, and ${\sf K}+\b^n\bot$, for $n\geq 1$, where $\b^n$ denotes $n$ consecutively occurrences of $\b$.
\end{itemize}

As mentioned in Section \ref{sec:Fixed points in arithmetic}, some authors formalize the Surprise Test Paradox in Peano arithmetic and interpret the surprise in terms of deducibility (see e.g. \cite{Shaw,Fitch1964,KritchmanRaz}). Regarding Solovay arithmetical completeness of $\GL$, it is possible to formalize the paradox in $\GL$, see Appendix \ref{Appendix:Surprise test paradox in GL}.

\subsection{Modal $\mu$-calculus}\label{sec:Modal mu-calculus}
Modal $\mu$-calculus \cite{BradfieldStirling2007,Jager2010,Kozen1983} is a logic used extensively in certain areas of computer science. It was first introduced by Kozen in \cite{Kozen1983}. The language of the modal $\mu$-calculus is an extension of the language of modal logic with variable binding operator $\mu p$ (the least fixed point operator). The expression $\mu p.A$ is intended to present, by the Knaster-Tarski theorem, the least fixed point of the operator naturally associated with the formula $A(p)$.

\begin{theorem}[Knaster-Tarski]
Given a set $S$, any monotone operator $\Phi$
\[\Phi:\mathcal{P}(S) \r \mathcal{P}(S)\]
within the ordering $(\mathcal{P}(S),\subseteq)$ has a least fixed point and a greatest fixed point.
\end{theorem}

Formulas of modal mu-calculus are constructed by the following grammar:
  \[ A::= p~|~\neg A~|~A\wedge A~|~\b A~|~\mu p.A,\]
  provided that every free occurrence of $p$ is \textit{positive} in $A$, i.e. every occurrence of $p$ in $A$ occurs within the scope of an even number of negations (in this case we say that $A$ is $p$-positive). The system ${\sf K(\mu)}$ is obtained from the basic modal logic {\sf K} by adding the \textit{closure axiom} scheme and the \textit{induction rule}, provided that $A$ is $p$-positive:

\begin{description}
     \item[$\mu$-CL.] $A(\mu p. A(p))\leftrightarrow \mu p.A(p)$,
         \item[$\mu$-IND.] from $\vdash A(B)\r B$ infer $\vdash \mu p. A(p)\r B$.
\end{description}
The greatest fixed point operator is defined as follows:
  \[\nu p.A := \neg\mu p. \neg A(\neg p).\]
The background modal logic {\sf K} can be extended to other modal logics to obtain consistent extensions of ${\sf K(\mu)}$, such as ${\sf S4(\mu)}$ and ${\sf S5(\mu)}$. For more detailed exposition see \cite{BradfieldStirling2007,Jager2010}.

Mardaev in \cite{Mardaev1994} showed that special family of $p$-positive formulas has fixed points in $\s$:
for any \s-formula $A(p,\b q_1,\ldots,\b q_n)$ in which $p$ is positive, there exists a formula $D(\b q_1,\ldots,\b q_n)$ such that
$$\s\vd D(\b q_1,\ldots,\b q_n)\leftrightarrow A(D(\b q_1,\ldots,\b q_n),\b q_1,\ldots,\b q_n).$$

Mardaev also shows that every $p$-positive $\Sigma$-formula $\varphi(p,\bar{q})$ has a fixed point in ${\sf K4}$ (cf. \cite{Mardaev1992}), and every $p$-positive $\Pi$-formula $\varphi(p,\bar{q})$ has a fixed point in $\GL$ (cf. \cite{Mardaev1993a}). For the definitions of $\Sigma$- and $\Pi$-formulas and a survey of Mardaev's results see \cite{Mardaev2007}.

It is worth noting that the Knower Paradox cannot be formalized in the modal $\mu$-calculus, since we need the fixed point of the formula $A(p)=\neg \b p$ in which $p$ is not positive. However, Halpern and Moses \cite{HalpernMoses1986} formalized some versions of the Surprise Test Paradox within a fixed point modal logic similar to  modal $\mu$-calculus.

\subsection{Connections}
The connection between $\GL$ and modal $\mu$-calculus has been studied by authors. Van Benthem in \cite{Benthem2006} showed that $\GL$ can be faithfully embedded into  ${\sf K(\mu)}$. He also showed that

\begin{theorem}
The logics ${\sf K(\mu)} + \b(\b A\r A)\r\b A$ and ${\sf K(\mu)} + \b A\r\b\b A + \mu p. \b p $ are equivalent.
\end{theorem}

Since in the $\mu$-calculus upward well-foundedness\footnote{A relation $R$ is upward well-founded if there exists no infinite sequence of worlds $w_1,w_2,\ldots$ such that $w_i R w_{i+1}$ for $i\geq 1$. It is known that upward well-foundedness is not definable in basic modal logic (see e.g. \cite{Benthem2006}).} can be expressed by the formula $\mu p.\b p$, the above theorem says that upward well-foundedness is modally definable (by the L\"{o}b axiom) together with transitivity.

Visser in \cite{Visser2005} gave another interpretation of $\GL$ into ${\sf K(\mu)}$. He also proved a generalized fixed point property for \GL.

\begin{definition}
A $\GL$-formula $\varphi(p)$ is semi-positive in $p$ if all non-modalized occurrences of $p$ are positive.
\end{definition}
\begin{theorem}[\cite{Visser2005}]
Any formula $\varphi(p)$ that is semi-positive in $p$ has an fixed point in \GL. Moreover, if all occurrences of $p$ in $\varphi(p)$ are positive, then the fixed point of $\varphi(p)$ is minimal.
\end{theorem}

Finally, Alberucci and Facchini in \cite{AlberucciFacchini2009} showed that the modal $\mu$-calculus over $\GL$ collapses to $\GL$. They also gave a new proof for the de Jongh-Sambin Fixed Point Theorem in \GL.

\section{Justification logics}\label{sec:Justification logics}
 The language of justification logics is an
extension of the language of propositional logic by the formulas
of the form $t:F$, where $F$ is a formula and $t$ is a
justification term. \textit{Justification terms} (or
\textit{terms} for short) are built up from (justification)
variables $x, y, z, \ldots$ (possibly with subscript) and (justification) constants $a,b,c,\ldots$ (possibly with subscript) using several operations depending on the logic: (binary) application  `$\cdot$', (binary) sum `$+$', (unary) verifier `$!$', (unary) negative verifier `$?$', and (unary) weak negative verifier `$\bar{?}$'.
 Justification formulas  are constructed by the following grammar:
\[ A::= p~|~\bot~|~\neg A~|~A\wedge A~|~A\vee A~|~A\rightarrow A~|~t:A,\]
where $p$ is a propositional variable and $t$ is a justification term.

We now begin with describing the axiom schemes and rules of the basic
justification logic {\sf J}, and continue with other justification
logics. The basic justification logic {\sf J} is the weakest
justification logic we shall be discussing. Other
justification logics are obtained by adding certain axiom schemes
to {\sf J}.

\begin{definition}\label{def: justification logics}
Axioms schemes of {\sf J} are:
\begin{description}
\item[Taut.] All propositional tautologies,
 \item[Sum.]  $s:A\r (s+t):A~,~s:A\r (t+s):A$,
 \item[jK.] $s:(A\r B)\r(t:A\r (s\cdot t):B)$.
\end{description}
Other justification logics are obtained by adding the following axiom schemes to {\sf J} in various combinations:
\begin{description}
\item[jT.]  $t:A\r A$.
\item[jD.]  $t:\perp \r \perp$.
 \item[j4.]  $t:A\r !t:t:A$,
\item[jB.]  $\neg A\r\bar{?} t:\neg t: A$.
 \item[j5.]  $\neg t:A\r ?t:\neg t:A$.
  \end{description}
 All justification logics have the inference rule Modus Ponens, and the \textit{Iterated Axiom Necessitation}  rule:
\begin{description}
\item[IAN.]
  $\vdash c_{i_n}:c_{i_{n-1}}:\ldots:c_{i_1}:A$, where $A$ is an axiom instance of the logic, $c_{i_j}$'s
are arbitrary justification constants and $n\geq 1$.
\end{description}
 \end{definition}

The language of each justification logic includes those operations on terms that are present in its axioms. Moreover, as in the case of modal logic, the name of each justification logic is indicated by the list of its axioms.  For example, ${\sf JT4}$ is the extension of ${\sf J}$ by axioms $jT$ and $j4$, in the language containing term operations $\cdot$, $+$, and $!$. {\sf JT4} is usually called the logic of proofs $\LP$.

\begin{remark}
The rule IAN can be replaced by the following rule, called \textit{Axiom Necessitation}  rule, in those justification logics that contain axiom j4:
\begin{description}
\item [AN.] $\vdash c:A$, where $A$ is
an axiom instance of the logic and $c$ is an arbitrary justification constant.
\end{description}
Artemov used this rule in his formulation of the logic of proofs $\LP$. We will use this rule in the formulation of quantified logic of proofs in Section \ref{sec:QLP}.
\end{remark}

\begin{definition}
\begin{enumerate}
\item Given a justification logic $\JL$, the \textit{total constant specification} $\mathcal{TCS}$ of $\JL$ is the set of \textit{all} formulas of the form $c_{i_n}:c_{i_{n-1}}:\ldots:c_{i_1}:A$, where $n\geq 1$, $A$ is an axiom instance of $\JL$ and $c_{i_j}$'s are arbitrary justification constants.
\item A \textit{constant specification} $\CS$ for $\JL$  is a subset of the total constant specification of $\JL$.
\item  A constant specification $\CS$ is \textit{axiomatically appropriate} if for each axiom instance $A$ of $\JL$ there is a constant $c$ such that $c:A\in\CS$, and if $F\in\CS$ then $c:F\in\CS$ for some constant $c$.
\end{enumerate}
 \end{definition}

Let ${\sf JL}_\CS$ be the fragment of ${\sf JL}$ where the Iterated Axiom Necessitation rule only produces formulas from the given $\CS$. Thus ${\sf JL}_\emptyset$ denotes the fragment of ${\sf JL}$ without the Iterated Axiom Necessitation rule. Note that the total constant specification $\mathcal{TCS}$ is axiomatically appropriate.

The deduction theorem and substitution lemma holds in all justification logics.

\begin{theorem}[Deduction Theorem]
For a set of formulas $S$, we have $\JL_\CS, $ $S, A\vdash B$ if and only if $\JL_\CS, S\vdash A\r B$.
\end{theorem}
\begin{lemma}[Substitution Lemma]\label{lemma: substitution lemma JL}
 If $\JL_\mathcal{TCS},S\vdash A,$ then for every justification variable $x$ and justification term $t$, we
have $\JL_\mathcal{TCS},S[t/x]\vdash A[t/x]$, where $A[t/x]$ is  the result of simultaneously replacing all occurrences of variable $x$ in $A$ by term $t$. The same holds if $\JL_\mathcal{TCS}$ is replaced by $\JL_\emptyset$.
\end{lemma}

The following lemma was first proved by Artemov in \cite{A2001}.
\begin{lemma}[Lifting Lemma]\label{lem:lifting lemma}
Given an axiomatically appropriate constant specification $\CS$ for \JL, if $$\JL_\CS, A_1,\ldots,A_n\vd F,$$ then for some  justification term $t(x_1,\ldots,x_n)$ and justification variables $x_1,\ldots,x_n$
$$\JL_\CS,x_1:A_1,\ldots,x_n:A_n\vd  t(x_1,\ldots,x_n):F.$$
\end{lemma}
\begin{proof}
The proof is by induction on the derivation of
$F$. We have three base cases:
\begin{itemize}
 \item If $F$ is an axiom, then put $t:= c$, for a justification constant $c$, and use rule IAN to obtain $c:F$.
  \item If $F=A_i$, then put $t:=x_i$.
\end{itemize}
For the induction step we have two cases.
\begin{itemize}
 \item Suppose $F$ is obtained by Modus Ponens from $G\r F$ and $G$. By induction hypothesis, there are terms $u(x_1,\ldots,x_n)$
 and $v(x_1,\ldots,x_n)$ such that $u:(G\r F)$ and $v:G$ are derivable from $x_1:A_1,\ldots,x_n:A_n$. Then put $t:=u.v$ and use the axiom jK to obtain $u\cdot v:F$.
 \item Suppose $F$ is obtained from Iterated Axiom Necessitation rule, so $F=c_{i_n}:c_{i_{n-1}}:\ldots:c_{i_1}:B\in\CS$, for some axiom instance $B$. Then since $\CS$ is axiomatically appropriate, there is a justification constant $c$ such that $c:c_{i_n}:c_{i_{n-1}}:\ldots:c_{i_1}:B\in\CS$. Thus, put $t:=c$.\qed
\end{itemize}
\end{proof}

One of the important properties of justification logics is the \textit{internalization property}.

\begin{lemma}[Internalization Lemma]\label{lemma:Internalization Lemma}
Given an axiomatically appropriate constant specification $\CS$ for \JL, if ${\sf JL}_\CS \vdash F$, then there is a justification term
$t$ such that ${\sf JL}_\CS \vdash t:F$.
\end{lemma}
\begin{proof}
Special case of Lemma \ref{lem:lifting lemma}.\qed
\end{proof}

The following lemma is helpful. in the next section.
\begin{lemma}\label{lem: axiom D in JD}
${\sf JD}_\mathcal{TCS}$ proves $s:\neg A \r \neg t:A$, for every ${\sf JD}$-formula $A$ and terms $s,t$.
\end{lemma}
\begin{proof}
\begin{enumerate}
\item $c:(\neg A \r (A\r \bot))$, by IAN
\item $c:(\neg A \r (A\r \bot))\r (s:\neg A \r c \cdot s:(A\r \bot))$, an instance of jK
\item $s:\neg A \r c \cdot s:(A\r \bot)$, from 1, 2 by MP
\item $c \cdot s:(A\r \bot) \r (t:A \r (c\cdot s) \cdot t:\bot)$, an instance of jK
\item $(c\cdot s) \cdot t:\bot \r \bot$, an instance of jD
\item $c \cdot s:(A\r \bot) \r (t:A \r \bot)$, from 4, 5 by propositional reasoning
\item $s:\neg A \r \neg t:A$, from 3, 6 by propositional reasoning.\qed
\end{enumerate}
\end{proof}

From the above lemma we obtain $t:A \r \neg s: \neg A$ in ${\sf JD}$ which is an analog of modal axiom  D, $\b A\r\Diamond A$.

In the sequel, we will state the precise connection between modal and justification logics. For comparison, axioms and rules of \LP~and \s~are given in Table 1.

\begin{table}[ht]
\centering
\renewcommand{\arraystretch}{1.5}
\begin{tabular}{|c|c|}
\hline
Modal logic \s & Logic of proofs \LP \\
\hline\hline
 $~\b(A\r B)\r(\b A\r\b B)~$ &~ $s:(A\r B)\r (t:A\r s\cdot t:B)~$ \\
 $\b A\r A$ & $t:A\r A$\\
 $\b A\r \b\b A$ & $t:A\r !t:t:A$ \\\hline
 &$s:A\vee t:A\r (s+t):A$\\ \hline\hline
\AXC{}\noLine
 \UIC{$\vd A$}\RightLabel{$(Nec)$}
 \UIC{$\vd\b A$}
 \noLine\UIC{}\DP
  &
  \AXC{}\noLine
  \UIC{\textrm{A is an axiom instance}}\RightLabel{$(AN)$}
  \UIC{$\vd c: A$}\noLine
  \UIC{}\DP\\  \hline
 \end{tabular}
 \vspace{0.2cm}
\caption{The correspondence between \s~and \LP.}
\end{table}

\begin{definition}\label{def: forgetful projection}
The forgetful projection $\circ$ is a mapping from the set of
justification formulas into the set of modal formulas, defined
recursively as follows: $p^\circ :=p$, $(\bot)^\circ := \bot$, $\circ$ commutes with propositional connectives, and $(t:A)^\circ := \Box A^\circ$.
For a set of justification formulas $S$, let
$S^\circ=\{F^\circ~|~F\in S\}$.
\end{definition}
 \begin{theorem} [Realization Theorem,  \cite{A2001,A2008,Brezhnev2000,GoetschiKuznets2012,Rubtsova2006}]\label{thm: realization of S4-LP}
$\JL^\circ={\sf ML}$.
\end{theorem}

If $\JL^\circ={\sf ML}$, then $\JL$ is called the justification counterpart of ${\sf ML}$. The justification counterpart of G\"{o}del-L\"{o}b provability logic  {\sf EGL} is introduced in \cite{Ghari2012-Thesis}. {\sf EGL} is an extension of {\sf J4} by the explicit L\"{o}b axiom schema:\footnote{Another explicit version of L\"{o}b axiom is considered in \cite{Ghari2011-AIMC,Ghari2012-Thesis} of the form $s:(t: A\r A)\r lob(s,t): A$, in the extended language of {\sf J4} by binary term operator $lob(.,.)$.}
$$s:(t: A\r A)\r t: A.$$
It is proved that ${\sf EGL}^\circ =\GL$. This, together with Solovay's arithmetical completeness of $\GL$ (\cite{Solovay1976}), implies the arithmetical provability completeness of {\sf EGL},
\[{\sf EGL} \hookrightarrow \GL \hookrightarrow \PA, \]
in which $t:A$ is interpreted as ``$A$ is provable in $\PA$."

\section{Fixed points in justification logics}\label{sec:Fixed points in justification logics}
In this section we study the effect of adding fixed points to justification logics. Some justification logics inherits the fixed point property from $\GL$. For example,  the \textit{logic of proofs and provability} {\sf GLA} (see \cite{ArtemovNogina2004,Nogina2006}) is such a logic. {\sf GLA} has axioms and rules of $\GL$ and $\LP$ (in their joint language), together with axioms $t:A\r \b A$,
 $\neg t:A\r\b\neg t:A$, $t:\b A\r A$, and the reflection rule: from $\vdash \b A$, infer $\vdash A$. It is obvious that every formula $A(p,\bar{q})$ in the language of modal logic in which $p$ is modalized has a fixed point in {\sf GLA}.

It is worth noting that a fixed point theorem for two operation-free logics of proofs was given by Stra{\ss}en in \cite{Strassen1994}. The systems considered there only use variables as terms and have no term operations (such as $\cdot$ and $+$) and hence are not of interest for current paper.
\subsection{Justification logics with fixed point operators}\label{sec:JL(FP)}
Suppose {\sf JL} is a propositional justification logic defined over a language ${\cal L}$ for which the internalization and substitution lemma could be proved. An occurrence of a propositional variable $p$ is called \textit{justified} in the formula $A(p,q_1,\ldots,q_n)$ if $p$ occurs in the scope of a justification operator $:$. Let ${\cal L}({\sf FP})$ be the extension of ${\cal L}$ by n-ary \textit{fixed point operators} $\delta_A(q_1,\ldots,q_n)$ for each ${\cal L}$-formula $A(p,q_1,\ldots,q_n)$ in which $p$ is justified. The \textit{fixed point extension} of justification logic ${\sf JL}$, denoted ${\sf JL(FP)}$, in the language ${\cal L}({\sf FP})$ is an extension of ${\sf JL}$ by \textit{fixed point axiom} schemes
\[ \delta_A(B_1,\ldots,B_n) \leftrightarrow A(\delta_A(B_1,\ldots,B_n),B_1,\ldots,B
_n),\]
where $B_1,\ldots,B_n$ are ${\cal L}({\sf FP})$-formulas.

The definitions of constant specification and ${\sf JL(FP)}_\CS$ are similar to those of $\JL$. It is easy to verify that the deduction theorem holds in ${\sf JL(FP)}_\CS$, for arbitrary $\CS$, the substitution lemma holds in ${\sf JL(FP)}_\mathcal{TCS}$ and ${\sf JL(FP)}_\emptyset$, and the internalization lemma holds in ${\sf JL(FP)}_\CS$, for axiomatically appropriate $\CS$.

Next the analogs of the Knower and the Believer Paradoxes are formulated in the framework of justification logics.

\begin{theorem}\label{thm: Knower Paradox-justification case}
Let {\sf JL} be a propositional justification logic which contain the axiom scheme
\begin{description}
  \item[jT.] $t:A\r A$,
  \end{description}
then ${\sf JL(FP)}_\mathcal{TCS}$ is inconsistent.
\end{theorem}
\begin{proof}
 In the following we derive a contradiction in ${\sf JL(FP)}$ using the fixed point axiom for the formula $A(p)=\neg x:p$.
\begin{enumerate}
  \item $\delta \leftrightarrow \neg x: \delta$, by fixed point axiom where $\delta=\delta_A$
  \item $x:\delta\r \neg \delta$, from 1 by propositional reasoning
  \item $x:\delta\r \delta$, an instance of jT
  \item $\neg x:\delta$, from 2, 3 by propositional reasoning
  \item $\delta$, from 1, 4 by propositional reasoning
  \item $t:\delta$, from 5 by the internalization lemma
  \item $t:\delta\r \neg \delta$, from 2 by the substitution lemma
  \item $\neg \delta$, from 6, 7 by MP
  \item $\bot$, from 5, 8.\qed
\end{enumerate}
\end{proof}

\begin{theorem}
Let {\sf JL} be a propositional modal logic which contain the axiom schemes
\begin{description}
  \item[jD.] $\neg t:\bot$,
  \item[j4.] $t:A\r !t:t:A$,
  \end{description}
then ${\sf JL(FP)}_\mathcal{TCS}$ is inconsistent.
\end{theorem}
\begin{proof}
Consider the fixed point axiom for the formula $A(p)=\neg x:p$.
\begin{enumerate}
  \item $\delta \leftrightarrow \neg x: \delta$, by fixed point axiom where $\delta=\delta_A$
  \item $\delta\r \neg x:\delta$, from 1 by propositional reasoning
  \item $t:(\delta\r \neg x:\delta)$, from 2 by the internalization lemma
  \item $t:(\delta\r \neg x:\delta) \r (x:\delta\r t\cdot x:\neg x:\delta)$, an instance of jK
  \item $x:\delta\r t\cdot x:\neg x:\delta$, from 3, 4 by MP
  \item $t\cdot x:\neg x:\delta\r \neg !x:x:\delta$, by lemma \ref{lem: axiom D in JD}
  \item $x:\delta\r \neg !x:x:\delta$, from 5, 6 by propositional reasoning
  \item $x:\delta\r !x:x:\delta$, an instance of j4
  \item $\neg x:\delta$, from 7, 8 by propositional reasoning
  \item $\delta$, from 1, 9 by MP
  \item $t:\delta$, from 10 by the internalization lemma
  \item $t:\delta\r \neg \delta$, from 1 by the substitution lemma
  \item $\neg \delta$, from 11, 12 by MP
  \item $\bot$, from 10, 13.\qed
\end{enumerate}
\end{proof}

For example the logics ${\sf JT(FP)}_\mathcal{TCS}$, ${\sf LP(FP)}_\mathcal{TCS}$,  and ${\sf JD4(FP)}_\mathcal{TCS}$ are inconsistent.
In the following, we will show that ${\sf J4(FP)}$ is consistent. First we extend the definition of forgetful projection to $\mathcal{L}({\sf FP})$.

\begin{definition}\label{def: forgetful projection for diagonal extension}
The forgetful projection $\circ$ of Definition \ref{def: forgetful projection} is extended to the language with fixed point operators as follows: $$(\delta_A(B_1,\ldots,B_n))^\circ = \delta_{A^\circ}(B_1^\circ,\ldots,B_n^\circ).$$
\end{definition}

\begin{lemma}
Given a constant specification $\CS$ for ${\sf J4(FP)}$ and a formula $F$ in the language of ${\sf J4(FP)}$, if ${\sf J4(FP)}_\CS\vdash F$, then ${\sf K4(FP)}\vdash F^\circ$.
\end{lemma}
\begin{proof}
By induction on the proof of $F$ in ${\sf J4(FP)}_\CS$. We only check the case that $F$ is a fixed point axiom. Suppose $p$ is justified in the ${\sf J4}$-formula $A(p,\bar{q})$, and $F$ is
\[ \delta_A(B_1,\ldots,B_n) \leftrightarrow A(\delta_A(B_1,\ldots,B_n),B_1,\ldots,B
_n),\]
for ${\sf J4(FP)}$-formulas $B_1,\ldots,B_n$. Hence,
\[ F^\circ=\delta_{A^\circ}(B_1^\circ,\ldots,B_n^\circ) \leftrightarrow A^\circ(\delta_{A^\circ}(B_1^\circ,\ldots,B_n^\circ),B_1^\circ,\ldots,B_n^\circ).\]
 Since $p$ is justified in $A(p,\bar{q})$, $p$ is modalized in $A^\circ(p,\bar{q})$. Therefore, for the fixed point operator $\delta_{A^\circ}(\bar{q})$ we have
\[ {\sf K4(FP)} \vdash \delta_{A^\circ}(C_1,\ldots,C_n) \leftrightarrow A^\circ(\delta_{A^\circ}(C_1,\ldots,C_n),C_1,\ldots,C_n),\]
for every ${\sf K4(FP)}$-formulas $C_1,\ldots,C_n$. Thus, for ${\sf K4(FP)}$-formulas $B_1^\circ,\ldots,B_n^\circ$ we have
\[ {\sf K4(FP)} \vdash \delta_{A^\circ}(B_1^\circ,\ldots,B_n^\circ) \leftrightarrow A^\circ(\delta_{A^\circ}(B_1^\circ,\ldots,B_n^\circ),B_1^\circ,\ldots,B_n^\circ).\]
Therefore, $ {\sf K4(FP)} \vdash F^\circ$.\qed
\end{proof}

\begin{corollary}
Given a constant specification $\CS$ for ${\sf J4(FP)}$, ${\sf J4(FP)}_\CS$ is consistent.
\end{corollary}
\begin{proof}
Suppose ${\sf J4(FP)}_\CS$ is inconsistent, ${\sf J4(FP)}_\CS\vd \bot$. Then, ${\sf K4(FP)}\vd \bot^\circ$, and thus ${\sf K4(FP)}\vd \bot$, which means ${\sf K4(FP)}$ is inconsistent, which is a contradiction. \qed
\end{proof}

Other fixed point extensions of justification logics are not known to be consistent.

The following lemma is an {\sf EGL} counterpart of the fixed point equation (\ref{eq:Lob formula fixed point}), $\b(\b A\r A) \leftrightarrow \b A$, in \GL.

\begin{lemma}
 The ${\sf EGL}$-formula $F(p)=c\cdot t:(p\rightarrow A)$ has the fixed point $t:A$ in ${\sf EGL}_\mathcal{TCS}$, where $c:(A\r(t:A\r A))$.
\end{lemma}
\begin{proof}
\begin{enumerate}
\item $c:(A\r(t:A\r A))$, by IAN
\item $c\cdot t:(t:A\rightarrow A) \r t:A$, an instance of explicit L\"{o}b axiom
\item $c:(A\r(t:A\r A)) \r (t:A \r c\cdot t:(t:A\r A))$, an instance of jK
\item $t:A \r c\cdot t:(t:A\r A)$, from 1, 3 by MP
\item $c\cdot t:(t:A\rightarrow A) \leftrightarrow t:A$, from 2, 4 by propositional reasoning.\qed
\end{enumerate}
\end{proof}
\subsection{Justification $\mu$-calculus}
In the previous section we showed that the fixed point extension of some of the justification logics is inconsistent. In this section, we try to find a consistent fixed point extension of these logics, based on $\mu$-calculus.

First we introduce a justification version of the modal mu-calculus ${\sf K(\mu)}$, called $\m$.
The language of $\m$ is an expansion of the language of ${\sf J}$.
Terms of $\m$ are defined similar to the terms of ${\sf J}$ by the following grammar:
\[ t::= x_i~|~c_i~|~t\cdot t~|~t+t.\]
Formulas of $\m$ are formed by the following grammar:
\[ A::= p~|~\neg A~|~A\wedge A~|~t:A~|~\mu p.A,\]
where $p$ is a propositional variable, $t$ is a term, and in $\mu p.A$ the formula $A$ is $p$-positive. We also assume to have the usual definitions
for $\neg$, $\vee$, $\r$ and $\leftrightarrow$ as logical
connectives in the above language. $\bot$ is defined as $A\wedge\neg A$ for some $\m$-formula $A$. In addition $\nu p.A$ is defined as before:
\[\nu p.A := \neg\mu p. \neg A(\neg p).\]

$\m$ is axiomatizable by adjoining to the basic justification logic {\sf J} the closure axiom scheme $\mu$-CL and the induction rule $\mu$-IND from Section \ref{sec:Modal mu-calculus}.
The definitions of constant specification and $\m_\CS$ are similar to those of $\JL$.

It is easy to verify that the deduction theorem holds in $\m_\CS$, for arbitrary $\CS$, and the substitution lemma holds in $\m_\mathcal{TCS}$ and $\m_\emptyset$. Note that the internalization lemma does not hold in $\m$ in its general form.

\begin{lemma}[Internalization Lemma for $\m$]\label{lemma:Internalization Lemma mu}
Given an axiomatically appropriate constant specification $\CS$ for $\m$, if $\m$-formula $F$ is derivable in $\m_\CS$ without the use of rule $\mu$-IND, then there is a justification term
$t$ such that $\m_\CS \vdash t:F$.
\end{lemma}

 Next by translating $\m$ into modal $\mu$-calculus ${\sf K(\mu)}$ we show that $\m$ is consistent.

\begin{definition}\label{Def: forgetful projection mu}
The forgetful projection $\circ$ of Definition \ref{def: forgetful projection} is extended to the language of $\m$ as follows: $(\mu p.A)^\circ := \mu p. A^\circ$.
\end{definition}

\begin{lemma}\label{lem:forgetful projection mu}
Given a constant specification $\CS$ for $\m$, for every $\m$-formula $A$, if $\m_\CS\vd A$, then ${\sf K(\mu)}\vd A^\circ$.
\end{lemma}
\begin{proof}
By induction on the proof of formula $A$ in $\m_\CS$.
\end{proof}

\begin{corollary}\label{cor:K(mu) is consistent}
Given a constant specification $\CS$ for $\m$, $\m_\CS$ is consistent.
\end{corollary}
\begin{proof}
Suppose $\m$ is inconsistent, $\m_\CS\vd \bot$. Thus, by Lemma \ref{lem:forgetful projection mu} we have ${\sf K(\mu)}\vd \bot^\circ$, or ${\sf K(\mu)}\vd \bot$, which would contradict the consistency of ${\sf K(\mu)}$. \qed
\end{proof}

Since the modal part of the $\mu$-calculus ${\sf K(\mu)}$ can be consistently extended to other modal logics, such as {\sf T}, {\sf S4}, {\sf S5}, we can consistently add other justification axioms to $\m$. For example, ${\sf LP(\mu)}$ is obtained by adding the term operator $!$ to the language of $\m$ and the axioms jT and j4 to $\m$, and ${\sf JT45(\mu)}$ is obtained by adding  the term operator $?$ to the language of ${\sf LP(\mu)}$ and the axioms j5 to ${\sf LP(\mu)}$. The proof of consistency of ${\sf JT45(\mu)}$ and ${\sf LP(\mu)}$ is similar to the proof of Corollary \ref{cor:K(mu) is consistent}.
\section{Fixed points in the quantified logic of proofs}\label{sec:QLP}
So far we have only considered the propositional justification logics. There are two known ways to introduce quantifiers in the logic of proofs:

\begin{enumerate}
 \item Quantifiers over objects (which the objects are interpreted as elements of the domain of models). Artemov and Yavorskaya \cite{ArtemovSidon2001} proved that first order logic of proofs equipped  with an arithmetical provability semantics is not axiomatizable. Without the arithmetical provability semantics an axiomatic system for first order logic of proofs is given in \cite{ArtemovSidon2011}.

 \item Quantifiers over justifications or proofs. Yavorsky \cite{Yavorsky2001} proved that the logic of proofs with quantifiers over proofs equipped  with an arithmetical provability semantics (in which the justification assertions are interpreted by multi-conclusion version of the G\"{o}del proof predicate in \PA) is not axiomatizable. Without the arithmetical provability semantics an axiomatic system for logic of proofs with quantifiers over justifications is given  by Fitting in \cite{Fitting2004,Fitting2008}.
 \end{enumerate}

In the following we recall the Fitting's quantified logic of proofs $\QLP$. Then we introduce fixed point extension of $\QLP$, and formalize the Knower and the Surprise Test Paradoxes in $\QLP$.
\subsection{Axiom system and basic properties of \QLP}
Instead of simple justification constants, Fitting uses \textit{primitive proof terms}. In fact, the language of $\QLP$ contains a countable set of \textit{primitive function symbols} of various arities. Primitive function symbols with arity 0 are indeed  justification constants. A primitive (proof) term is a term of the form $f^n(x_1,\ldots,x_n)$, or simply $f(x_1,\ldots,x_n)$, where $f^n$ is a primitive function symbol of arity $n$ and $x_1,\ldots,x_n$ are justification variables.

Let us first describe the language and axiom system of $\QLP^-$ (a subsystem of {\sf QLP}), and then those of $\QLP$. Justification terms and formulas of $\QLP^-$ are constructed by the following grammars:
\[ t::= x_i~|~f^n(x_1,\ldots,x_n)~|~t\cdot t~|~t+t~|~!t,\]
\[ A::= p~|~\bot~|~\neg A~|~A\wedge A~|~A\vee A~|~A\rightarrow A~|~t:A~|~(\forall x)A~|~(\exists x) A,\]
where $i,n$ are non-negative integers, $x,x_i$'s are justification variables, $t$ is a justification term, and $f^n(x_1,\ldots,x_n)$ is a primitive proof term. Note that the universal quantifier quantifies over justification variables.
The definition of free and bound occurrences of variables and substitution of variables by terms are as in the first order logic.

Axioms and rules of $\QLP^-$ are a combination of axioms and rules of first order logic and logic of proofs $\LP$. More precisely, axioms of $\QLP^-$ are:
\begin{description}
\item[Taut.] All tautologies of propositional logic,
\item[Q1.] $(\forall x)A(x)\r A(t)$, where $t$ is free for $x$ in $A(x)$,
    \item[Q2.] $(\forall x)(A\r B(x))\r(A\r (\forall x)B(x))$, where $x$ does not occur free in $A$,
    \item[Q3.] $A(t) \r (\exists x) A(x)$, where $t$ is free for $x$ in $A(x)$,
    \item[Q4.] $(\forall x)(A(x)\r B)\r((\exists x) A(x)\r B)$, where $x$ does not occur free in $B$,
 \item[jK.] $s:(A\r B)\r(t:A\r (s\cdot t):B)$,
   \item[jT.] $t:A\r A$,
    \item[j4.] $t:A\r !t:t:A$,
     \item[Sum.] $s:A\r (s+t):A~,~s:A\r (t+s):A$,
\end{description}
Rules of $\QLP^-$ are Modus Ponens, Generalization, and Axiom Necessitation rule:
\[ \di{\f{A ~~~ A\r B}{B}MP}, \quad \di{\f{A}{(\forall x)A}}Gen, \quad
\di{\f{A~\mbox{is an axiom instance}}{f(x_1,\ldots,x_n):A}AN},\]
where $f(x_1,\ldots,x_n)$ is a primitive term.

Fitting's quantified logic of proofs is an extension of $\QLP^-$ by first adding a binary term operator, called \textit{uniform verifier}, as follows: if $t$ is a term and $x$ is a justification variable, then $(t \forall x)$ is a term. The occurrence of $x$ in $(t \forall x)$ is considered to be bound.  Thus justification terms of $\QLP$ are constructed by the following grammar:

\[ t::= x_i~|~f^n(x_1,\ldots,x_n)~|~t\cdot t~|~t+t~|~!t~|~(t \forall x).\]
$\QLP$ in addition has the following axiom, called \textit{uniformity formula} UF:

$$(\exists y)y:(\forall x) t:A \r (t\forall x):(\forall x)A,$$
provided that $y$ does not occur free in $t$ or $A$, and \textit{Quantified Necessitation} rule:

\[\di{\f{A}{(\exists x)x:A}}qNec\]
provided that $x$ does not occur free in  $A$.

\begin{definition}
\begin{enumerate}
\item The \textit{total primitive term specification} for $\QLP$ ($\QLP^-$) is the set of \textit{all} formulas of the form $f(x_1,\ldots,x_n):A$, where $A$ is an axiom instance of $\QLP$ ($\QLP^-$) and $f(x_1,\ldots,x_n)$ is a primitive term.
\item A \textit{primitive term specification} $\mathcal{F}$ for $\QLP$ ($\QLP^-$)  is a subset of the total primitive term specification for $\QLP$ ($\QLP^-$).
\item  A primitive term specification $\mathcal{F}$ for $\QLP$ ($\QLP^-$)  is called \textit{axiomatically appropriate} if for each axiom instance $A$ of $\QLP$ ($\QLP^-$) there is a primitive term $f(x_1,\ldots,x_n)$ such that $f(x_1,\ldots,x_n):A\in\mathcal{F}$.
\end{enumerate}
 \end{definition}

Let $\QLP_\mathcal{F}$ ($\QLP^-_\mathcal{F}$) be the fragment of $\QLP$ ($\QLP^-$) where the Axiom Necessitation rule only produces formulas from $\mathcal{F}$. Thus $\QLP_\emptyset$ ($\QLP^-_\emptyset$) denotes the fragment of $\QLP$ ($\QLP^-$) without the Axiom Necessitation rule.

It is easy to show the following (see \cite{Fitting2008}).

\begin{theorem}
Given a primitive term specification $\mathcal{F}$, the Justified Universal Generalization rule:
\[\di{\f{t:A(x)}{(t\forall x):(\forall x)A(x)}}JUG\]
is admissible in $\QLP_\mathcal{F}$.
\end{theorem}
\begin{proof}
Suppose $\QLP_\mathcal{F} \vdash t:A(x)$. By Gen, we get $\QLP_\mathcal{F} \vdash (\forall x) t:A(x)$. Then by qNec, we get $\QLP_\mathcal{F} \vdash (\exists y) y:(\forall x) t:A(x)$, for a variable $y$ where does not occur free in $t$ or $A$. By axiom UF and MP, we obtain $\QLP_\mathcal{F} \vdash (t \forall x):(\forall x) A(x)$ as desire. \qed
\end{proof}

In fact, the original axiomatization of $\QLP$ in \cite{Fitting2008} has the rule JUG instead of Gen and qNec.

It is worth noting that a restricted version of the rule qNec is admissible in $\QLP^-$.

\begin{theorem}
For the total primitive term specification $\mathcal{F}$, the following rule is admissible in $\QLP^-_\mathcal{F}$:
\[\di{\f{A~\mbox{is an axiom instance}}{(\exists x)x:A}}\]
where $x$  does not occur free in  $A$.
\end{theorem}
\begin{proof}
Suppose A is an axiom instance of $\QLP^-_\mathcal{F}$. By AN, we get $\QLP^-_\mathcal{F} \vdash c:A$. Then, by axiom $c:A \r (\exists x)x:A$, where $x$  does not occur free in  $A$, we have $\QLP_\mathcal{F} \vdash (\exists x) x:A$ as desire. \qed
\end{proof}

\begin{lemma}[Internalization Lemma for $\QLP$]\label{lemma:Internalization Lemma QLP}
Let $\mathcal{F}$ be an axiomatically appropriate primitive term specification. If $\QLP_\mathcal{F} \vdash F$, then there is a justification term
$t$ such that $\QLP_\mathcal{F} \vdash t:F$.
\end{lemma}
\begin{proof}
The proof is by induction on the derivation of $F$ in $\QLP_\mathcal{F}$ (similar to the proof of Lemma \ref{lem:lifting lemma}). The only new cases are when $F$ is obtained by the rules Gen and qNec.

For the case of Gen, suppose $F=(\forall x) A$ is obtained from $A$. By the induction hypothesis, there is a term $u$ such that $u:A$ is derivable in $\QLP_\mathcal{F}$. Using Gen, we have $(\forall x) u:A$. Then by qNec, we get $(\exists y) y: (\forall x)u:A$, for a variable $y$ where does not occur free in $u$ or $A$. By axiom UF and MP, we obtain $(u \forall x):(\forall x) A(x)$. Thus it suffices to put $t:=(u \forall x)$.

For the case of qNec, suppose $F=(\exists x) x:A$ is obtained from $A$, where $x$  does not occur free in  $A$. By the induction hypothesis, there is a term $u$ such that $u:A$ is derivable in $\QLP_\mathcal{F}$. By axiom j4 and MP, we get $!u:u:A$. Since $u:A \r (\exists x) x:A$ is an instance of axiom Q3, by AN we get $c:( u:A \r (\exists x) x:A)$ for some constant $c$. Now from the latter and $!u:u:A$ and axiom jK, we get $c\cdot !u :(\exists x) x:A$. Thus it suffices to put $t:=c\cdot !u$. \qed
\end{proof}

As you can see from the above proof in order to obtain an internalized version of the Generalization rule we need the uniformity formula. Therefore, in general the internalization property does not hold in $\QLP^-$. However we have a restricted form of the internalization lemma.

\begin{lemma}[Internalization Lemma for $\QLP^-$]\label{lemma:Internalization Lemma QLP^-}
Let $\mathcal{F}$ be an axiomatically appropriate primitive term specification. If $F$ is derivable in $\QLP^-_\mathcal{F}$ without the use of rule Gen, then there is a justification term $t$ such that $\QLP^-_\mathcal{F} \vdash t:F$.
\end{lemma}

Fitting \cite{Fitting2008} gives a translation from propositional modal logic {\sf S4} into $\QLP$ as follows.
\begin{definition}\label{def:translation exists}
The mapping $\exists$ from $\s$-formulas into $\QLP$-formulas is defined  as follows:  $p^\exists=p$, $\bot^\exists=\bot$, $\exists$ commutes with propositional connectives,  $(\b A)^\exists= (\exists x) x:A^\exists$.
For a set $S$ of modal formulas, let $S^\exists=\{F^\exists~|~F\in S\}$.
\end{definition}

\begin{theorem}[\cite{Fitting2008}]
For every $\s$-formula $A$,
\[\s\vdash A\qquad \Leftrightarrow\qquad \QLP\vdash A^\exists\]
\end{theorem}

The same correspondence, $\Box A = (\exists x) x:A$, is considered by Yavorsky in \cite{Yavorsky2001} for his quantified logic of proofs {\sf qLP}.
%

If we interpret the modality $\b$ as knowledge (i.e. $\b F$ read as ``$F$ is known"), then the translation $\b A \rightleftharpoons (\exists x) x:A$ gives the following (related) interpretations of knowledge:

\begin{enumerate}
\item Proof-based interpretation of knowledge, where knowledge of $A$  means ``there is a formal proof for $A$" or ``$A$ is provable" (indeed this interpretation is provided by Yavorsky's {\sf qLP}).
\item Evidence-based interpretation of knowledge, where knowledge of $A$  means ``there is an evidence (or justification) for $A$" (indeed this interpretation is provided by Fitting's $\QLP$).
\end{enumerate}

Although the evidence-based interpretation of knowledge satisfies the principle of knowledge veracity, $(\exists x) x:\varphi\r\varphi$, the proof-based interpretation does not. In fact, since mathematical proofs can be considered as certain evidences, proof-based knowledge implies evidence-based knowledge, but not vice versa. Note also that having evidence-based (or proof-based) knowledge of a statement is only a sufficient condition for having knowledge of that statement.
\subsection{Fixed point extensions of \QLP}

Next let us turn to the fixed point extension of $\QLP$. If we define the fixed point extension of $\QLP$ as the one for propositional justification logics in Section \ref{sec:JL(FP)} (i.e. fixed point axioms are defined for formulas with justified occurrences of propositional variables), then Theorem \ref{thm: Knower Paradox-justification case} already shows that this fixed point extension is inconsistent. According to Definition \ref{def:translation exists}, it is natural to define fixed point axioms  for formulas with boxed occurrences of propositional variables (this is similar to one defined for fixed point extensions of modal logics in Section \ref{sec:ML(FP)}).

The propositional variable $p$ is called \textit{$\exists$-justified} in the $\QLP$-formula $A(p,\bar{q})$ if each occurrences of $p$ in $A(p,\bar{q})$ is in the scope of $(\exists x)x:...$, for some variable $x$. To put it otherwise, if all occurrences of $(\exists x)x:...$, for some variable $x$, in the $\QLP$-formula $A(p,\bar{q})$ is replaced by $\b$, then $p$ is $\exists$-justified in $A(p,\bar{q})$ if all occurrences of $p$ are in the scope of $\b$ in $A(p,\bar{q})$. For example, $p$ is $\exists$-justified in the following formula

\[ A(p) = (\exists x) x:p \vee (\exists y) y:\neg p.\]

Now extend the language of $\QLP$ by fixed point operators $\delta_A(\bar{q})$, for each $\QLP$-formula $A(p,\bar{q})$ in which $p$ is $\exists$-justified. The fixed point extension of $\QLP$, denoted by ${\sf QLP(FP)}$ is obtained by adding the fixed point axioms:
\[ \delta_A(\bar{B}) \leftrightarrow A(\delta_A(\bar{B}),\bar{B}) \]
where $p$ is $\exists$-justified in $A(p,\bar{q})$, and $\bar{B}$ is a list of ${\sf QLP(FP)}$-formulas.

In the rest of this section it is useful to consider intermediate systems between $\QLP$ and ${\sf QLP(FP)}$. Given the $\QLP$-formula $A(p,\bar{q})$ in which $p$ is $\exists$-justified, first extend the language of $\QLP$ by single fixed point operator $\delta_A(\bar{q})$, and then define the logic
$$\QLP(\delta_A(\bar{B}) \leftrightarrow A(\delta_A(\bar{B}),\bar{B}))$$
 to be the extension of $\QLP$ with single fixed point axiom $\delta_A(\bar{B}) \leftrightarrow A(\delta_A(\bar{B}),\bar{B})$. This notion is useful when we are dealing with the extension of $\QLP$ with a single particular fixed point axiom.

The definition of primitive term specification, ${\sf QLP(FP)}_\mathcal{F}$, and ${\sf QLP}(F)_\mathcal{F}$, for fixed point axiom $F$, is similar to those of $\QLP$. All the definitions given above can be stated for $\QLP^-$ instead of $\QLP$ as well.
\subsection{The Knower Paradox in \QLP}
Using the idea of evidence-based interpretation of knowledge, the Knower Paradox was  redeveloped in \cite{Arlo-CostaKishida2009,Dean2014,DeanKurokawa2014} within Fitting's quantified logic of proofs {\sf QLP}. The Knower Paradox
\[D \leftrightarrow \neg K(\gn{D}) \quad \mbox{or} \quad D \leftrightarrow \neg \b D,\]
 is expressible in $\QLP$ by the formula:
\[D \leftrightarrow \neg (\exists x)x:D. \]
We give the proof of the Knower in a fixed point extension of $\QLP$.\footnote{In contrast to the proofs of the Knower paradox in $\QLP$ given in  \cite{Dean2014,DeanKurokawa2009,DeanKurokawa2014}, the proof of Theorem \ref{thm:Knower Paradox in QLP} (similar to the original proof of the Knower paradox in Theorem \ref{thm:Knower, simple case}) is expressed without any hypothesis.}

\begin{theorem} [The Knower Paradox in \QLP]\label{thm:Knower Paradox in QLP}
  Let $\delta$ be the fixed point operator of the formula $A(p)=\neg (\exists x)x:p$. Then
\[ \QLP(\delta \leftrightarrow \neg (\exists x)x:\delta)_\emptyset\]
 is inconsistent.
\end{theorem}
\begin{proof}
Recall that $\QLP(\delta \leftrightarrow \neg (\exists x)x:\delta)_\emptyset$ does not contain rule AN.
 \begin{enumerate}
 \item $\delta \leftrightarrow \neg (\exists x)x:\delta$, fixed point axiom
  \item $\neg(\exists x)x:\delta \r  \delta$, from 1 by propositional reasoning
 \item $(\exists x)x:\delta \r \neg \delta$, from 1 by propositional reasoning
 \item $x:\delta \r \delta$, an instance of axiom jT
 \item $(\forall x)(x:\delta \r \delta)$, from 4 by Gen
 \item $(\forall x)(x:\delta \r \delta) \r ((\exists x) x:\delta \r \delta)$, an instance of axiom Q4
 \item $(\exists x) x:\delta \r \delta$, from 5, 6 by MP
 \item $\neg(\exists x) x:\delta$, from 3, 7 by propositional reasoning
 \item $\delta$, from 2, 8 by MP
 \item $t:\delta$, from 9 by the internalization lemma (Lemma \ref{lemma:Internalization Lemma QLP})
 \item $t:\delta \r (\exists x) x:\delta$, an instance of axiom Q3
\item $(\exists x) x:\delta$, from 10, 11 by MP
 \item $\bot$, from 8, 12.\qed
 \end{enumerate}
\end{proof}

Note that the above proof cannot be proceeded in $\QLP^-$, since we use the internalization lemma (in step 10) on a provable sentence $\delta$ which uses Gen (in step 5) in its proof. Moreover, the formula $(\exists x) x:\delta$ in step 12 can be directly derived from $\delta$ (in step 9) by qNec.

\begin{theorem} [The Knower Paradox in $\QLP^-$]\label{thm:Knower Paradox in QLP^-}
  The logic
\[ \QLP^-(\delta \leftrightarrow \neg (\exists x)x:\delta)_\emptyset\]
 is consistent.
\end{theorem}
\begin{proof}
See Appendix \ref{Appendix:Semantics of QLP}.\qed
\end{proof}

The Knower Paradox was also studied in \cite{DeanKurokawa2009} in the framework of Fitting's quantified logic of proofs from \cite{Fitting2004,Fitting2006}. The quantified logic of proofs presented in \cite{Fitting2004,Fitting2006} has the same language of $\QLP$  but instead of axiom UF and rule qNec it contains the following axiom (called \textit{Uniform Barcan Formula}):
 $$(\forall x)t:A(x) \r (t\forall x):(\forall x) A(x),$$
 where $x$ does not occur free in $t$. Dean and Kurokawa presented some arguments against this axiom, and suggest to resolve the Knower Paradox by abandoning it.

Theorem \ref{thm:Knower Paradox in QLP} implies that ${\sf QLP(FP)}$ is inconsistent.

\begin{corollary}
${\sf QLP(FP)}_\emptyset$ is inconsistent.
\end{corollary}

Dean and Kurokawa in \cite{Dean2014,DeanKurokawa2009} give an arithmetical interpretation for $\QLP^-$, and show the arithmetic soundness of $\QLP^-$: the interpretation of every formula provable in $\QLP^-$ is true in the standard model of arithmetic. Nonetheless, it is not clear how fixed point operators of ${\sf QLP^-(FP)}$ are related to fixed point sentences of $\PA$. Having shown this connection, it is perhaps possible to prove that  ${\sf QLP^-(FP)}$ is consistent. We leave the details for future work.

\subsection{The Surprise Test Paradox in \QLP}\label{sec:The Surprise Test Paradaox in QLP}
In this section we analyze the Surprise Test Paradox in the framework of $\QLP$, when we have an evidence-based interpretation of knowledge in mind.

Since the test is supposed to be surprise for students (and not for non-students), it is helpful to consider a multi-agent version of $\QLP$. Suppose $\mathcal{A}=\{1,2,\ldots,n\}$ is a set of agents. The language of multi-agent quantified logic of proofs $\QLP_n$ is similar to $\QLP$ with the difference that formulas are constructed by the following grammar:
\[ A::= p~|~\bot~|~\neg A~|~A\wedge A~|~A\vee A~|~A\rightarrow A~|~t:_s A~|~(\forall x)A~|~(\exists x) A,\]
where $s\in \mathcal{A}$. The justification assertion $t:_s A$ is read ``the agent $s$ considers $t$ as a justification (or reason) for $A$." Axioms and rules of $\QLP_n$ are those of $\QLP$ where $``:"$ is replaced by $``:_s"$ everywhere, for arbitrary agent $s$ in $\mathcal{A}$. One can show that $\QLP_n$ is consistent, by giving a translation from the language of $\QLP_n$ to the language of $\QLP$ that maps $t:_s A$ into $t:A$.

Now it is natural to interpret a surprise statement $A$ as a statement for which there is no justification or reason. This can be expressed in $\QLP$ as $\neg (\exists x)x: A$. More precisely,

\begin{center}
{``$A$ is a surprise for $s$" $\rightleftharpoons$ ``$s$ has no reason for $A$" $\rightleftharpoons$ ``$\neg (\exists x)x:_s A$"}
\end{center}

The above definition of surprise is not claimed to be an exceptionless definition for a surprise event in everyday life; rather, it seems it is  natural in the context of this paradox, i.e., a test is surprise for a student if she does not know based on any evidence the date of the test.

To keep the notation simple, we consider a class with only one student $s$ and use  $``:"$ instead of $``:_s"$. In the sequel we formalize various versions of the Surprise Test Paradox in $\QLP$ and its fixed point extensions.

First consider the Kaplan-Montague  self-reference one-day case  of the paradox, the Examiner Paradox:
\begin{quote}
 ``Unless you know this statement to be false, you will have a test tomorrow, but you can't know from this statement that you will have a test tomorrow."
\end{quote}
This sentence was formalized in Section \ref{sec:Fixed points in arithmetic} as follows:

\begin{equation}\label{eq:Self-reference Examiner Paradox again}
D \leftrightarrow ( K(\gn{\neg  D})\vee( E\wedge \neg K(\gn{D\r E}))),
\end{equation}

  where $E$ denotes the sentence ``you will have a test tomorrow." The sentence (\ref{eq:Self-reference Examiner Paradox again}) is expressed in ${\sf QLP(FP)}$ as follows:

\[\delta \leftrightarrow [ (\exists x)x:\neg \delta\vee( E\wedge \neg (\exists x)x:(\delta\r E))]\]

where $\delta=\delta_A(E)$ is the fixed point operator of the formula

$$A(p,E)=(\exists x)x:\neg p\vee( E\wedge \neg (\exists x)x:(p\r E)).$$

We show that the Examiner Paradox leads to a contradiction in $\QLP$.

  \begin{theorem} [The Examiner Paradox in \QLP]\label{thm:The Examiner Paradox in QLP}
  The logic
\[\QLP(\delta \leftrightarrow [ (\exists x)x:\neg \delta\vee( E\wedge \neg (\exists x)x:(\delta\r E))])_\emptyset\]
is inconsistent.
 \end{theorem}

\begin{proof}
\begin{enumerate}
\item $\delta \leftrightarrow [ (\exists x)x:\neg \delta\vee( E\wedge \neg (\exists x)x:(\delta\r E))]$, fixed point axiom
\item $x:\neg\delta \r \neg\delta$, an instance of axiom jT
 \item $(\forall x)(x:\neg\delta \r \neg\delta)$, from 2 by Gen
 \item $(\forall x)(x:\neg\delta \r \neg\delta) \r ((\exists x) x:\neg\delta \r \neg\delta)$, an instance of axiom Q4
 \item $(\exists x)x:\neg \delta \r\neg\delta$, from 3, 4 by MP
\item $\delta\r\neg(\exists x)x:\neg \delta$, from 5 by propositional reasoning
\item $\delta\r E\wedge\neg(\exists x)x:(\delta\r E)$, from 1, 6 by propositional reasoning
\item $\delta\r E$, from 7 by propositional reasoning
\item $\delta\r\neg(\exists x)x:(\delta\r E)$, from 7 by propositional reasoning
\item $(\exists x)x:(\delta\r E)\r\neg\delta$, from 9 by propositional reasoning
\item $(\exists x) x:(\delta\r E)$, from 8 by qNec
\item $\neg \delta$, from 10, 11 by MP
\item $(\exists x)x:\neg \delta$, from 12 by qNec
\item $\delta$, from 1, 13 by propositional reasoning
\item $\bot$, from 12, 14.\qed
\end{enumerate}
\end{proof}

Note that the above proof cannot be proceeded in $\QLP^-$, since we use the rule qNec (in steps 11 and 13). Moreover, It also give another proof for the inconsistency of ${\sf QLP(FP)}$.

 \begin{theorem} [The Examiner Paradox in $\QLP^-$]\label{thm:The Examiner Paradox in QLP^-}
 The logic
  \[\QLP^-(\delta \leftrightarrow [ (\exists x)x:\neg \delta\vee( E\wedge \neg (\exists x)x:(\delta\r E))])_\emptyset\]
  is consistent.
\end{theorem}
\begin{proof}
See Appendix \ref{Appendix:Semantics of QLP}.\qed
\end{proof}

Theorems \ref{thm:The Examiner Paradox in QLP}, \ref{thm:The Examiner Paradox in QLP^-} gives a solution to the Kaplan-Montague self-reference one-day case of the Surprise Test Paradox. The self-reference $n$-day case is similar.

Next consider the following self-reference one-day case of the paradox:

\begin{quote}
``You will have a test tomorrow, but you can't know from this statement that you will have a test tomorrow."
\end{quote}

This sentence can be expressed in ${\sf QLP(FP)}$ as follows:

\[\delta \leftrightarrow [  E\wedge \neg (\exists x)x:(\delta\r E)]\]

where $\delta=\delta_A(E)$ is the fixed point operator of the formula

$$A(p,E)= E\wedge \neg (\exists x)x:(p\r E).$$

We show that this announcement can not be fulfilled.
  \begin{theorem}\label{thm:self one-day surprise without unless QLP}
  \[\QLP(\delta \leftrightarrow [  E\wedge \neg (\exists x)x:(\delta\r E)])_\emptyset\vdash \neg\delta\]
 \end{theorem}
\begin{proof}
\begin{enumerate}
\item $\delta \leftrightarrow [  E\wedge \neg (\exists x)x:(\delta\r E)]$, fixed point axiom
\item $\delta\r E$, from 1 by propositional reasoning
\item $\delta\r\neg(\exists x)x:(\delta\r E)$, from 1 by propositional reasoning
\item $(\exists x)x:(\delta\r E)\r\neg\delta$, from 3 by propositional reasoning
\item $(\exists x) x:(\delta\r E)$, from 2 by qNec
\item $\neg \delta$, from 4, 5 by MP. \qed
\end{enumerate}
\end{proof}

\begin{theorem}\label{thm:self one-day surprise without unless QLP^-}
    \[\QLP^-(\delta \leftrightarrow [  E\wedge \neg (\exists x)x:(\delta\r E)])_\emptyset\not\vdash \neg\delta\]
\end{theorem}
\begin{proof}
See Appendix \ref{Appendix:Semantics of QLP}.\qed
\end{proof}

Theorems \ref{thm:self one-day surprise without unless QLP}, \ref{thm:self one-day surprise without unless QLP^-} gives a solution to the self-reference one-day case of the Surprise Test Paradox. The self-reference $n$-day case is similar.

Now consider the self-reference two-day case of the paradox:

\begin{quote}
``A teacher announces that there will be exactly one surprise test on Wednesday or Friday next week, but you can't know from this statement the date of the test."
\end{quote}

This sentence can be expressed in ${\sf QLP(FP)}$ as follows:

\[\delta \leftrightarrow [  E_1\wedge \neg (\exists x)x:(\delta\r E_1)]\veebar[  E_2\wedge \neg (\exists x)x:(\delta\wedge\neg E_1\r E_2)]\]

where $E_1$ and $E_2$  denote the sentences ``you will have a test on Wednesday" and ``you will have a test on Friday" respectively, and $\delta=\delta_A(E_1,E_2)$ is the fixed point operator of the formula

$$A(p,E_1,E_2)= [  E_1\wedge \neg (\exists x)x:(p\r E_1)]\veebar[  E_2\wedge \neg (\exists x)x:(p\wedge\neg E_1\r E_2)],$$
and $\veebar$ denotes the exclusive disjunction.

\begin{theorem}\label{thm:self-reference tow-day Surprise Test Paradox in QLP}

  \[{\sf QLP}(\delta \leftrightarrow [  E_1\wedge \neg (\exists x)x:(\delta\r E_1)]\veebar[  E_2\wedge \neg (\exists x)x:(\delta\wedge\neg E_1\r E_2)])_\emptyset\vdash \neg\delta\]
\end{theorem}
\begin{proof}
\begin{enumerate}
\item $\delta \leftrightarrow [  E_1\wedge \neg (\exists x)x:(\delta\r E_1)]\veebar[  E_2\wedge \neg (\exists x)x:(\delta\wedge\neg E_1\r E_2)]$, fixed point axiom
\item $\delta\wedge\neg E_1\r E_2$, from 1 by propositional reasoning
\item $(\exists x)x:(\delta\wedge\neg E_1\r E_2)$, from 2 by qNec
\item $\delta\r E_1$, from 1, 3 by propositional reasoning
\item $(\exists x)x:(\delta\r E_1)$, from 4 by qNec
\item $\neg \delta$, from 1, 5 by propositional reasoning.\qed
\end{enumerate}
\end{proof}

\begin{theorem}\label{thm:self-reference tow-day Surprise Test Paradox in QLP^-}

  \[\QLP^-(\delta \leftrightarrow [  E_1\wedge \neg (\exists x)x:(\delta\r E_1)]\veebar[  E_2\wedge \neg (\exists x)x:(\delta\wedge\neg E_1\r E_2)])_\emptyset \not\vdash \neg\delta\]
\end{theorem}
\begin{proof}
See Appendix \ref{Appendix:Semantics of QLP}.\qed
\end{proof}

Finally consider the non-self-reference one-day case of the paradox as follows:
\begin{quote}
``You will have a test tomorrow that will take you by surprise, i.e. you can't know it beforehand."
\end{quote}

As Sorensen proposed in \cite{Sorensen1984,Sorensen2014} the above statement is an epistemic blindspot for the students.\footnote{Binkley \cite{Binkley1968} presented the same analysis but in the doxastic modal logic {\sf KD4}, and conclude that the Surprise Test Paradox belongs to the same family as Moor's paradox. For a related discussion see also Quine's argument in \cite{Quine1953}.} An statement $A$ is an \textit{epistemic blindspot} for person $s$ if and only if $A$ is true but not known by $s$, i.e. $A\wedge \neg K_s A$. (Such sentences are also called \textit{pragmatically paradoxical}.) This statement can be expressed in $\QLP$ by

\begin{equation}\label{eq:blindspot in QLP}
E\wedge \neg (\exists x)x:_s E.
\end{equation}

We show that it is provable in $\QLP^-$ that the teacher's announcement (\ref{eq:blindspot in QLP}) is a blindspot for the student $s$.

\begin{theorem}
\[\QLP^-_\mathcal{F} \vdash\neg (\exists y) y:_s[  E\wedge \neg (\exists x)x:_s E]\]
where $\mathcal{F}=\{c:_s (E\wedge \neg (\exists x)x:_s E \r E)\}$.
\end{theorem}
\begin{proof}
Let $F=E\wedge \neg (\exists x)x:_s E$. We show that $\neg (\exists y) y:_s F$ is provable in $\QLP^-_\mathcal{F}$.
\begin{enumerate}
\item $F \r E$, a propositional tautology
\item $c:_s (F \r E)$, from 1 by AN
\item $c:_s (F \r E) \r (y:_s F \r c \cdot y :_s E)$, an instance of axiom jK
\item $y:_s F \r c\cdot y:_s E$, from 2, 3 by MP
\item $c\cdot y:_s E \r (\exists x)x:_s E$, an instance of axiom Q3
\item $y:_s F \r (\exists x)x:_s E$, from 4, 5 by propositional reasoning
\item $F\r \neg (\exists x)x:_s E$, a propositional tautology
\item $y:_s F \r F$, an instance of axiom jT
\item $y:_s F \r \neg (\exists x)x:_s E$, from 7, 8 by propositional reasoning
\item $\neg y:_s F$, from 6, 9 by propositional reasoning
\item $(\forall y) \neg y:_s F$, from 10 by Gen
\item $\neg (\exists y) y:_s F$, from 11 by reasoning in first order logic. \qed
\end{enumerate}
\end{proof}

But note that (\ref{eq:blindspot in QLP}) is not necessarily a blindspot for others, i.e. we could consistently have $(\exists y) y:_{s'}[  E\wedge \neg (\exists x)x:_s E]$ for a non-student person $s'$.

\section{Conclusion}
We have presented several fixed point extensions of justification logics: extensions by fixed point operators, and extensions by least fixed points. There remained one  problem here. Is there a justification logic with the fixed point property, namely a justification logic for which a fixed point theorem can be proved in its original language? A complete affirmative answer to this question is not expected, since the rule substitution of equivalents SE does not hold in justification logics (as noted in the Introduction).

We have also presented fixed point extensions of Fitting's quantified logic of proofs, and formalize the Knower Paradox and various versions of the Surprise Test Paradox in these extensions. By interpreting a surprise statement as a statement for which there is no justification, we give a solution to the self-reference version of the  Surprise Test Paradox. Our analysis of theses paradoxes presumes evidence-based interpretation of knowledge, that is to say knowledge of a fact means there is an evidence for that fact. With regard to this fact, the paradoxes could be resolved in $\QLP^-$, i.e by restricting the axioms and rules of quantified logic of proofs to those of first order logic and $\LP$.

In this respect, generally, we find both the rule qNec and the uniformity formula UF problematic. The rule qNec when formulated for an agent $s$ says: if $A$ is a theorem (of \QLP), then $s$ has an evidence for $A$. This implies that the agent knows every theorem, which means that the agent is logically omniscient. Hence, qNec is an acceptable rule only for idealized agents. However, we could suppose that students in the Surprise Test Paradox are idealized agents, and try to solve the puzzle in this setting. Thus, the only possible way to avoid the paradox is to reject the uniformity formula UF. This observation agrees with Dean-Kurokawa's analysis \cite{DeanKurokawa2014} which found the rule JUG suspicious.
\\

\noindent
{\bf Acknowledgments}\\

I would like to thank Ali Valizadeh and Mojtaba Mojtahedi  for their useful comments. This research was in part supported by a grant from IPM. (No. 93030416)
\appendix

\section{The Surprise Test Paradox in timed modal epistemic logics}\label{Appendix:Surprise test paradox in tMEL}
In order to formalize the $n$-day case of the surprise test paradox, most authors (see e.g. \cite{Binkley1968,Chow,KaplanMontague1960,Kripke2011}) uses temporalized knowledge operators such as $K_i (p)$, intended to mean that ``the student knows on day $i$ that $p$ is true."

Wang in \cite{Wang2011} presented a family of timed modal epistemic logics, ${\sf tMEL}$, with timed knowledge operators $K^i F$, for natural number $i$,  meaning that ``$F$ is known at time $i$." The interpretation of timed knowledge operators in {\sf tMEL} in  possible world semantics can be stated informally as follows: a proposition is known at time $i$ in a world $w$ if and only if it is true at all worlds accessible from $w$ and the first time that the truth of the proposition is accepted is before or at $i$ (for more details cf. \cite{Wang2011}).

Among various timed modal epistemic logics, we formalize the surprise test paradox within timed modal epistemic logic {\sf tK}.\footnote{In the notation of Wang \cite{Wang2011}, {\sf tK} is indeed a timed modal epistemic logic with \textit{comprehensive principal logical base}. We do not state the details of the definition of base and comprehensive principal logical base here, but let me only quote him: ``Agents with comprehensive logical bases are unrealistic. They know too much from the beginning." Thus, in this section we again suppose that students in the Surprise Test Paradox are idealized agents, and try to solve the puzzle in this setting.} Formulas of {\sf tK} are constructed by the following grammar:
\[ A::= p~|~\bot~|~\neg A~|~A\wedge A~|~A\vee A~|~A\rightarrow A~|~K^i A,\]
where $p$ is a propositional variable, and $i$ is a non-negative integer. Axioms of {\sf tK} are as follows:

\begin{description}
\item[Taut.] All tautologies of propositional logic,
\item[tK.] $K^i(A \r B)\r(K^j A \r K^k B)$, where $i,j<k$,
\item[Mon.] $K^i A \r K^j A$, where $i<j$.
\end{description}

Rules of inference are modus ponens and the followings:
\[\di{\frac{\vdash A}{\vdash K^i A \r K^j K^i A}\ (DE) \qquad \di{\frac{\vdash A}{\vdash K^i A}}}\ (E)\]
where in (DE) $i<j$.
Axiom tK is a temporal counterpart of modal axiom K. The monotonicity axiom Mon states that if one knows a statement at a time, then she knows it at any later time. This principle does not hold in everyday life of course.  In \cite{Wang2011} the rules (DE) and (E) are called respectively \textit{Deduction by Epistemization} and \textit{Epistemization}, and moreover (DE) is formulated as an axiom.

It is easy to verify that the deduction theorem holds in {\sf tK}. In addition, the following rule is admissible in {\sf tK}:
\[\di{\frac{\vdash A \r B}{\vdash K^i A \r K^j A}}\ (TKC)\]
for $i<j$. This rule is called \textit{Timed Knowledge Closure} in \cite{Wang2011}, and is used to claim that {\sf tK}-agents (where {\sf tK} are formulated by restricted forms of the rules (DE) and (E)) are not logical omniscience. The following theorem of {\sf tK} is also useful:

\begin{equation}\label{eq:K(A&B) in tK}
K^i(A \wedge B)\r K^j A \wedge K^j B, \quad\mbox{for}\ i<j
\end{equation}

Now let us consider the non-self-reference two-day case of the Surprise Test Paradox (the argument for $n$-day case is similar):
\begin{quote}
``A teacher announces that there will be exactly one surprise test on Wednesday or Friday next week."
\end{quote}
 Assume that teacher's announcement is taken place at time 0, Friday class meeting is at time 10, and Wednesday class meeting is at time 20. Teacher's announcement can be formalized as follows:
\[\varphi = (E_1 \veebar E_2) \wedge (E_1 \r \bigwedge_{0\leq i< 10} \neg K^i E_1) \wedge (E_2 \r (\bigwedge_{0\leq i<20} \neg K^i E_2 \wedge \bigwedge_{10<i\leq 20} K^i \neg E_1)),\]
where $E_1$ and $E_2$  denote the sentences ``you will have a test on Wednesday" and ``you will have a test on Friday" respectively, and $\veebar$ denotes the exclusive or. Let us formalize the first stage of the student's argument  in {\sf tK}, that is the test cannot be held on Friday. (Note that {\sf tK} is indeed a logic of belief, and hence in the following theorem $K^i A$ can be read as ``$A$ is believed at time $i$.")

\begin{theorem}\label{thm:surprise in tK}
$K^1(E_1 \vee E_2) , \varphi \vdash_{\sf tK} \neg E_2$.
\end{theorem}
\begin{proof}
\begin{enumerate}
\item $K^1(E_1 \vee E_2)$, premise
\item $\varphi$, premise
\item $E_2$, assumption
\item $\neg K^{12} E_2$, from 2, 3, by propositional reasoning
\item $K^{11} \neg E_1$, from 2, 3, by propositional reasoning
\item $(E_1 \vee E_2)\r (\neg E_1 \r E_2)$, a propositional tautology
\item $K^1(E_1 \vee E_2)\r K^2(\neg E_1 \r E_2)$, from 6 by rule (TKC)
\item $K^2(\neg E_1 \r E_2)$, from 1, 7 by MP
\item $K^{12} E_2$, from 5, 8 and axiom tK
\item $\bot$, from 4 and 9.
\end{enumerate}
Thus $K^1(E_1 \vee E_2) , \varphi,E_2 \vdash_{\sf tK} \bot$. By the deduction theorem, $K^1(E_1 \vee E_2) , \varphi \vdash_{\sf tK} \neg E_2$.\qed
\end{proof}

Now consider the timed logic of knowledge {\sf tT}, which is obtained from {\sf tK} by adding the following axiom:
\begin{description}
\item[tT.] $K^i A \r A$.
\end{description}
The first stage of the student's argument can be also formalized in {\sf tT} as follows.

\begin{theorem}\label{thm:surprise in tT}
$K^1 \varphi \vdash_{\sf tT} \neg E_2$.
\end{theorem}
\begin{proof}
\begin{enumerate}
\item $K^1 \varphi$, premise
\item $E_2$, assumption
\item $\varphi$, from 1 by axiom tT
\item $\varphi \r (E_1 \vee E_2)$, a propositional tautology
\item $K^1 \varphi \r K^2(E_1 \vee E_2)$, from 4 by rule (TKC)
\item $K^2(E_1 \vee E_2)$, from 1, 5  by MP
\item $\neg K^{12} E_2$, from 2, 3, by propositional reasoning
\item $K^{11} \neg E_1$, from 2, 3, by propositional reasoning
\item $(E_1 \vee E_2)\r (\neg E_1 \r E_2)$, a propositional tautology
\item $K^1(E_1 \vee E_2)\r K^2(\neg E_1 \r E_2)$, from 9 by rule (TKC)
\item $K^2(\neg E_1 \r E_2)$, from 6, 10 by MP
\item $K^{12} E_2$, from 8, 11 and axiom tK
\item $\bot$, from 7 and 12.
\end{enumerate}
Thus $K^1 \varphi,E_2 \vdash_{\sf tT} \bot$. By the deduction theorem, $K^1 \varphi \vdash_{\sf tT} \neg E_2$.\qed
\end{proof}

In order to formalize the student's argument completely and eliminate the remaining day (here Wednesday), we employ an extension of {\sf tT}. The timed modal epistemic logic {\sf tS4} is obtained by adding the positive introspection axiom to {\sf tT}:
\begin{description}
\item[t4.] $K^i A \r K^j K^i A$, where $i<j$.
\end{description}
It is not difficult to show that the following rule is admissible in {\sf tS4}:
\begin{equation}\label{eq:admissible rule in tS4}
\di{\frac{K^{i_1} A_1,\ldots,K^{i_n} A_n \vdash B}{K^{i_1} A_1,\ldots,K^{i_n} A_n \vdash K^j B}}
\end{equation}
where $i_1,\ldots,i_n < j$.
Now let us continue the student's argument and rule out the remaining day Wednesday.

\begin{theorem}\label{thm:surprise in tS4}
$K^1 \varphi \vdash_{\sf tS4} \neg E_1$.
\end{theorem}
\begin{proof}
\begin{enumerate}
\item $K^1 \varphi$, premise
\item $\neg E_2$, by Theorem \ref{thm:surprise in tT}
\item $\varphi$, from 1 by axiom tT
\item $E_1 \vee E_2$, from 3 by propositional reasoning
\item $E_1$, from 2, 4, by propositional reasoning
\item $K^2 E_1$, from 1-5 by rule (\ref{eq:admissible rule in tS4})
\item $E_1 \r \bigwedge_{0\leq i< 10} \neg K^i E_1$, from 3 by propositional reasoning
\item $E_1 \r \neg K^2 E_1$, from 7 by propositional reasoning
\item $K^2 E_1 \r \neg E_1$, from 8 by propositional reasoning
\item $\neg E_1$, from 6, 9, by MP\qed
\end{enumerate}
\end{proof}

In fact, lines 5 and 10 of the above proof get a contradiction from the assumption $K^1 \varphi$.
\begin{theorem}
$K^1 \varphi \vdash_{\sf tS4} \bot $.
\end{theorem}

In order to avoid the paradox Quine \cite{Quine1953} denies the premise $K^1\varphi$. Indeed, by the deduction theorem, we obtain $ \vdash_{\sf tS4} \neg K^1 \varphi$. This result agrees with Quine's opinion that students do not know that the announcement is true. On the other hand, if one accepts the principle that ``if someone is informed of a proposition, then he knows it" (as explicitly formulated by Wright and Sudbury in \cite{WrightSudbury}), then the above result leads to a contradiction. (An extended discussion of Quine's solution can be found in \cite{Cheung2013}.)

Kripke \cite{Kripke2011} also do not find Quine's solution satisfactory. He says ``But often, I think, you do \textit{know} something simply because a good teacher has told you so." He instead denies the monotonicity principle Mon. Quoting Kripke, ``You may know something now, but, on the basis of further evidence -- without any loss of evidence or forgetfulness -- be led to fall into doubt about it later." The monotonicity principle of knowledge also rejected by Sorensen \cite{Sorensen2014} and Williamson \cite{Williamson2000}. For example, Williamson says ``[the students'] memory of examinationless days would undermine their earlier knowledge of the truth of the announcement, like misleading evidence." The monotonicity principle is also considered by Binkley \cite{Binkley1968}, Wright and Sudbury  \cite{WrightSudbury} in the context of logics of belief. Wright and Sudbury rejected the monotonicity principle too, ``[the principle] is, manifestly, \textit{not} analytic of reasonable belief. Good reason to believe $p$ may lapse as more information becomes available; or stronger reason to believe the contrary may emerge" (cf. \cite{WrightSudbury}).

However, one can observe that the monotonicity principle Mon does not play any role in the students' argument in our formulations (see the proofs of Theorems \ref{thm:surprise in tT}, \ref{thm:surprise in tS4}).
There is another view in which the paradox is resolved by rejecting the KK principle, here axiom t4 (see e.g. \cite{Harrison1969, McLellandChihara1975}).\footnote{But this is criticized by Sorensen in \cite{Sorensen1988}.}

\section{The Surprise Test Paradox in G\"{o}del-L\"{o}b provability logic}\label{Appendix:Surprise test paradox in GL}
In this appendix we consider a version of the Surprise Test Paradox in which  a surprise event is interpreted in terms of deducibility. Specifically, a test is a surprise for a student if and only if the student cannot deduce logically beforehand the date of the test. One of the advantages of this meaning of surprise, as mentioned also by Fitch \cite{Fitch1964}, is that there is no need to use epistemological and temporal concepts.  We consider deducibility in $\PA$ and instead of formalizing the paradox in $\PA$ we give a formalization in the provability logic $\GL$.

Let us consider the following two-day version of the paradox (the argument for $n$-day version is similar):
\begin{quote}
``There will be exactly one test on Wednesday or Friday next week and its date will not be deducible from this statement."
\end{quote}
As discussed by Shaw \cite{Shaw} the paradox appears only if the teacher's announcement is self-reference. Regarding deducibility in \PA, Fitch \cite{Fitch1964} formalized the paradox in $\PA$. Since all valid statements about the provability predicate of $\PA$ can be described by the provability logic $\GL$ (see Theorem \ref{thm:Solovay Arithmetical Completeness GL}), we formalize the paradox in $\GL$. Teacher's announcement can be expressed in $\GL$ by the following fixed point equation:

\begin{equation}\label{eq:Fitch(7)}
D \leftrightarrow [(E_1 \wedge \neg\b(D \rightarrow E_1)) \veebar (E_2 \wedge\neg\b(D \wedge\neg E_1 \rightarrow E_2))],
\end{equation}

where $E_1$ and $E_2$  denote the sentences ``you will have a test on Wednesday" and ``you will have a test on Friday" respectively, and $\veebar$ denotes the exclusive or. The structure of the fixed point $D$ in (\ref{eq:Fitch(7)}) is not important here, but it can be computed by algorithms presented in \cite{Boolos1993, Smorynski1985, SambinValentini1982}. Following Fitch we show that the teacher's announcement $D$ in (\ref{eq:Fitch(7)}) is self-contradictory (we sketch only a few of the formal details here).

\begin{theorem}\label{thm:Fitch(7) is refutable in GL}
The fixed point $D$ in (\ref{eq:Fitch(7)}) is refutable in $\GL$.
\end{theorem}
\begin{proof}
The proof is as follows:
\begin{enumerate}
\item $D \leftrightarrow [(E_1 \wedge \neg\b(D \rightarrow E_1)) \veebar (E_2 \wedge \neg\b(D \wedge\neg E_1 \rightarrow E_2))]$, By De Jongh-Sambin Fixed Point Theorem \ref{thm:Fixed Point Theorem GL}
\item $D \rightarrow [(E_1 \wedge \neg\b(D \rightarrow E_1)) \vee (E_2 \wedge \neg\b(D \wedge\neg E_1 \rightarrow E_2))]$, from 1 by propositional reasoning
\item $D \wedge (\neg E_1 \vee \b(D \rightarrow E_1)) \r (E_2 \wedge \neg\b(D \wedge\neg E_1 \rightarrow E_2)$, from 2 by propositional reasoning
\item $D \wedge \neg E_1 \r E_2$, from 3 by propositional reasoning
\item $\b(D \wedge \neg E_1 \r E_2)$, from 4 by Nec
\item $D\wedge (\neg E_2 \vee \b(D \wedge\neg E_1 \rightarrow E_2)) \r (E_1 \wedge \neg\b(D \rightarrow E_1))$, from 2 by propositional reasoning
\item $D \r E_1$, from 5, 6 by propositional reasoning
\item $\b(D \r E_1)$, from 7 by Nec
\item $\neg E_1 \vee \b(D \r E_1)$, from 8 by propositional reasoning
\item $\neg E_2 \vee \b(D\wedge \neg E_1 \r E_2)$, from 5 by propositional reasoning
\item $\neg(E_1 \wedge \neg\b(D \rightarrow E_1)) \wedge \neg(E_2 \wedge \neg\b(D \wedge E_1 \rightarrow E_2))$, from 9, 10 by propositional reasoning
\item $\neg D$, from 11 by propositional reasoning.\qed
\end{enumerate}
\end{proof}

Since the teacher's announcement $D$ is refutable, it implies everything. Particularly, it implies that the test will not be held on Friday. Fitch tried to resolve the paradox by reinterpreting the surprise so that ``what is intended in practice is not that the surprise event will be a surprise \textit{whenever} it occurs, but only when it occurs on some day \textit{other than the last}".  His reformulation of the paradox can be expressed in $\GL$ by the following fixed point equation:

\begin{equation}\label{eq:Fitch(16)}
D \leftrightarrow [(E_1 \wedge \neg\b(D \rightarrow E_1)) \veebar (E_2 \wedge \neg\b(D  \rightarrow E_2))],
\end{equation}

He claims that the fixed point $D$ in (\ref{eq:Fitch(16)}) is ``apparently self-consistent" in \PA. Nevertheless, Kripke \cite{Kripke2011} (in a letter to Fitch) shows that the statement of Fitch's resolution is actually refutable in \PA, and therefore it is not resolved as stated by Fitch. The argument of Kripke can be formalized in $\GL$ as follows.

\begin{theorem}\label{thm:Fitch(16) is refutable in GL}
The fixed point $D$ in (\ref{eq:Fitch(16)}) is refutable in $\GL$.
\end{theorem}
\begin{proof}
The proof is as follows:
\begin{enumerate}
\item $D \leftrightarrow [(E_1 \wedge \neg\b(D \rightarrow E_1)) \veebar (E_2 \wedge \neg\b(D \rightarrow E_2))]$, By De Jongh-Sambin Fixed Point Theorem \ref{thm:Fixed Point Theorem GL}
\item $\neg D \r(D \r E_i)$, a propositional tautology, where $i=1, 2$
\item $\b\neg D \r\b(D \r E_i)$, from 2 by modal reasoning, where $i=1, 2$
\item $\b\neg D \r(\neg E_i \vee \b(D \rightarrow E_i))$, from 3 by propositional reasoning, where $i=1, 2$
\item $\b\neg D \r\neg(E_1 \wedge \neg\b(D \rightarrow E_1)) \wedge\neg (E_2 \wedge \neg\b(D \rightarrow E_2))$, from 4 by propositional reasoning
\item $\b\neg D\r \neg D$, from 1, 5 by propositional reasoning,
\item $\b(\b\neg D\r \neg D)$, from 6 by Nec
\item $\b(\b\neg D\r \neg D)\r \b\neg D$, an instance of L\"{o}b scheme
\item $\b\neg D$, from 7, 8 by MP
\item $\neg D$, from 6, 9 by MP.\qed
\end{enumerate}
\end{proof}

Fitch also show the relation between the following version of the Surprise Test Paradox and G\"{o}del's first incompleteness theorem:
\begin{quote}
``If this statement is deducible, then there will be exactly one test on Wednesday or Friday next week and its date will not be deducible from this statement."
\end{quote}
His reformulation of the paradox can be expressed in $\GL$ by the following fixed point equation:

\begin{equation}\label{eq:Fitch(18)}
D \leftrightarrow [\b D\r (E_1 \wedge \neg\b(D \rightarrow E_1)) \veebar (E_2 \wedge \neg\b(D \wedge\neg E_1 \rightarrow E_2))],
\end{equation}

\begin{theorem}\label{thm:Fitch(18)}
For the fixed point $D$ in (\ref{eq:Fitch(18)}) we have $\GL\vdash D \leftrightarrow \neg \b D$ and $\GL\not\vdash D$.
\end{theorem}
\begin{proof}
The proof is as follows:
\begin{enumerate}
\item $D \leftrightarrow [\b D\r(E_1 \wedge \neg\b(D \rightarrow E_1)) \veebar (E_2 \wedge \neg\b(D \wedge\neg E_1 \rightarrow E_2))]$, By De Jongh-Sambin Fixed Point Theorem \ref{thm:Fixed Point Theorem GL}
\item $\b D\r[D\r (E_1 \wedge \neg\b(D \rightarrow E_1)) \vee (E_2 \wedge \neg\b(D \wedge\neg E_1 \rightarrow E_2))]$, from 1 by propositional reasoning
\item $\b D\r (D \wedge \neg E_1 \r E_2)$, similar to stages 2-4 of the proof of Theorem \ref{thm:Fitch(7) is refutable in GL}
\item $\b\b D\r \b(D \wedge \neg E_1 \r E_2)$, from 3 by modal reasoning
\item $\b D \r \b\b D$, an instance of axiom 4
\item $\b D\r \b(D \wedge \neg E_1 \r E_2)$, from 4, 5 by propositional reasoning
\item $\b D\r(D \r E_1)$, similar to stages 5-7 of the proof of Theorem \ref{thm:Fitch(7) is refutable in GL}
\item $\b\b D\r\b(D \r E_1)$, from 7 by modal reasoning
\item $\b D\r\b(D \r E_1)$, from 8 by modal reasoning
\item $\b D\r \neg D$, similar to stages 7-12 of the proof of Theorem \ref{thm:Fitch(7) is refutable in GL}
\item $[\b D\r(E_1 \wedge \neg\b(D \rightarrow E_1)) \veebar (E_2 \wedge \neg\b(D \wedge\neg E_1 \rightarrow E_2))]\r D$, from 1 by propositional reasoning
\item $\neg \b D\r D$, from 11 by propositional reasoning
\item $D \leftrightarrow \neg \b D$, from 10, 12 by propositional reasoning.
\end{enumerate}
Finally note that if $\GL\vdash D$, then by Nec we get $\GL\vdash\b D$. On the other hand, from $\GL\vdash D \leftrightarrow \neg \b D$ we obtain $\GL\vdash\neg\b D$, a contradiction. \qed
\end{proof}

Thus, similar to G\"{o}del's sentence (see (\ref{eq:Godel sentence}) in the proof of Theorem \ref{thm:Incompleteness Theorem}), the fixed point $D$ in (\ref{eq:Fitch(18)}) states its own unprovability.

Kritchman and Raz \cite{KritchmanRaz} also show the relationship between the paradox and G\"{o}del's second incompleteness theorem. They formalized a version of the paradox and conclude that ``if the students believe in the consistency of $T+S$, the exam cannot be held on Friday [i.e. the last day], ... However, the exam can be held on any other day of the week because  [by the second incompleteness theorem] the students cannot prove the consistency of $T+S$." (where $T$ can be taken to be $\PA$, and $S$ is the statement of the teacher's announcement).

\section{Semantics of $\QLP$}\label{Appendix:Semantics of QLP}
Fitting in \cite{Fitting2008} presented Kripke-style possible world semantics for $\QLP$. His semantics is an extension of the $\LP$ possible world semantics of  \cite{Fitting2005} by first order semantics. In this part, we introduce models for $\QLP$ based on Mkrtychev models (M-model for short) for $\LP$ \cite{Mkrtychev1997}, that are actually single-world Fitting models \cite{Fitting2008} of $\QLP$. We use the same notations of Fitting models of $\QLP$ here.

\begin{definition}\label{def:M-model QLP^-}
An M-model for $\QLP^-_\mathcal{F}$ is a quadruple $\M=(D,\I,\E,\V)$ such that
\begin{enumerate}
\item  $D$ is the \textit{domain} of the model, a non-empty set (of \textit{reasons}).

\item  $\I$ is an  \textit{interpretation function}  mapping each term operation to an operator on $D$ as follows.
 \begin{enumerate}
 \item $\I$ assigns to each primitive function symbol $f$ of arity $n$ an $n$-place operator $f^\I : D^n \r D$; in particular, $\I$ assigns to each constant
$c$ a member $c^\I$ of $D$.
\item $\I$ assigns to the verifier operation $!$ a mapping $!^\I : D \r D$.
\item $\I$ assigns to the application operation $\cdot$ a binary operation $\cdot^\I : D \times D\r D$, and to the sum operation $+$ a binary operation $+^\I : D\times D\r D$.
 \end{enumerate}
Given a domain $D$, a valuation $v$ is defined as a mapping from justification variables to $D$. Given a domain $D$ and an interpretation $\I$, the valuation $v$ can be extended to all terms as follows (we use the notation $t^v$ instead of $v(t)$):
\begin{enumerate}
\item $x^v=v(x)$, for variable $x$,
\item $f(t_1,\ldots,t_n)^v = f^\I(t_1^v,\ldots,t_n^v)$, for primitive function symbol $f$ of arity $n$,
\item $(t \cdot s)^v = t^v \cdot^\I s^v$,
\item $(t + s)^v = t^v +^\I s^v$,
\item $(!t)^v =\ !^\I t^v$.
\end{enumerate}
\item  $\E$ is an \textit{evidence function}, that assigns to each reason  in the domain $D$ and to each valuation a set of formulas, and satisfying the following conditions. For all formulas $A$ and $B$, all reasons $r$ and $r'$ in $D$, and all valuations $v$:

 \begin{enumerate}
 \item Application: If $A\r B\in {\mathcal E}(r,v)$ and $A\in {\mathcal E}(r',v)$, then $B\in {\mathcal E}(r\cdot^\I r',v)$.
 \item Sum: ${\mathcal E}(r,v)\cup {\mathcal E}(r',v)\subseteq {\mathcal E}(r+^\I r',v)$.
  \item Proof checker: If $A\in {\mathcal E}(r,v)$, then $t:A\in {\mathcal E}(!^\I t^v,v)$.
  \item Primitive proof term: $A\in \E(f(x_1,\ldots,x_n)^v,v)$, for every $f(x_1,\ldots,x_n):A\in\mathcal{F}$.
   \end{enumerate}

\item $\V$ is a \textit{truth assignment}, i.e  a mapping from propositional variables to set of truth values $\{0,1\}$.
\end{enumerate}
\end{definition}

\begin{definition}\label{def:M-model QLP}
An M-model $\M=(D,\I,\E,\V)$ for $\QLP_\mathcal{F}$ is an M-model for  $\QLP^-_\mathcal{F}$  meeting the following conditions:
\begin{enumerate}
\item $\I$ assigns to the uniform verifier $\forall$ a mapping $\forall^\I : D \times D \r D$.
For an interpretation $\I$, and a valuation $v$ put $(t \forall x)^v = t^v \forall^\I x^v$.
\item If $A\in\E(t^{v {x \choose r}},v {x \choose r})$ for every $r\in D$, then $(\forall x)A\in\E((t \forall x)^v,v)$.
\item If $v$ and $w$ agree on the free variables of $A$ and $r\in D$, then $A\in\E(r,v)$ iff $A\in\E(r,w)$.
\end{enumerate}
\end{definition}

\begin{definition}
A valuation $w$ is an $x$-variant of a valuation $v$ if $w$ is identical to $v$ except possibly on $x$. The notation $v{x \choose r}$ denotes the $x$-variant of $v$ that maps $x$ to $r$.
 \end{definition}

\begin{definition}\label{def:forcing M-model QLP^-}
 The forcing relation $\Vdash_v$, for an M-model $\M=(D,\I,\E,\V)$ of $\QLP^-_\mathcal{F}$ and a valuation $v$ is defined in
the following way:
\begin{enumerate}
\item $\M\Vdash_v p$ if{f} $\V(p)=1$, for propositional variable $p$,
 \item $\M \not\Vdash_v\bot$,
 \item $\M\Vdash_v \neg A$ if{f} $\M\not\Vdash_v A$,
 \item $\M\Vdash_v A\vee B$ if{f} $\M\Vdash_v A$ or $\M\Vdash_v B$,
 \item $\M\Vdash_v A\wedge B$ if{f} $\M\Vdash_v A$ and $\M\Vdash_v B$,
 \item $\M\Vdash_v A\r B$ if{f} $\M\not\Vdash_v A$ or $\M\Vdash_v B$,
 \item $\M\Vdash_v (\forall x) A$ if{f} $\M\Vdash_{v {x \choose r}} A$, for every $r\in D$,
  \item $\M\Vdash_v (\exists x) A$ if{f} $\M\Vdash_{v {x \choose r}} A$, for some  $r\in D$,
 \item $\M\Vdash_v t:A$ if{f} $A\in {\mathcal E}(t^v,v)$ and $\M\Vdash_v A$.
\end{enumerate}
A formula $A$ is valid in an M-model $\M$ if $\M\Vdash_v A$ for every valuation $v$.
\end{definition}

\begin{definition}\label{def:strong M-model QLP}
An M-model $\M=(D,\I,\E,\V)$ for $\QLP_\mathcal{F}$ is called strong if it satisfies the following condition: for every valuation $v$, if $\M \Vdash_v A$, then $A\in\E(r,v)$, for some $r\in D$.
 \end{definition}

Let us compare this condition with the \textit{fully explanatory} condition of Fitting models of $\QLP$ \cite{Fitting2008}. The fully explanatory condition says that if a proposition is believed (at a state of a model), then it has a reason (at that state). While the condition of the above definition says that if a proposition is true (in a model), then it has a reason (in that model). This condition is plausible only for idealized agents.

It is not difficult to show the following.

\begin{lemma}\label{lem:truth agree on the free variables M-model QLP}
Suppose $\M=(D,\I,\E,\V)$ is an M-model for $\QLP_\mathcal{F}$. If the valuations $v$ and $w$ agree on the free variables of $A$, then $\M\Vdash_v A$ iff $\M\Vdash_w A$.
\end{lemma}

Since M-models of $\QLP^-$ are single-world Fitting models of $\QLP^-$, the soundness theorem of $\QLP^-$ with respect to M-models is a consequence of Fitting's soundness theorem in \cite{Fitting2008}.

\begin{theorem}[Soundness $\QLP^-$]
Every formula provable in $\QLP^-_\mathcal{F}$ is valid in every M-model of $\QLP^-_\mathcal{F}$.
\end{theorem}

The definition of interpretation of the uniform verifier, $\forall^\I$, given in Definition \ref{def:M-model QLP} is different from that given  by Fitting for Kripke-style models of $\QLP$ in \cite{Fitting2008}. Thus, we prove the soundness theorem for $\QLP$ in more details.

\begin{theorem}[Soundness $\QLP$]
Every formula provable in $\QLP_\mathcal{F}$ is valid in every strong M-model of $\QLP_\mathcal{F}$.
\end{theorem}
\begin{proof}
We only show the validity of the uniformity formula. The proof of the validity of other axioms are straightforward. Suppose $\M=(D,\I,\E,\V)$ is an strong M-model for $\QLP_\mathcal{F}$, and $v$ is an arbitrary valuation, such that $\M \Vdash_v (\exists y) y: (\forall x) t:A$, where $y$ does not occur free in $t$ or $A$. Then, for some $r_0\in D$, $\M \Vdash_{w}  y: (\forall x) t:A$, where $w=v {y \choose r_0}$. Hence, $\M \Vdash_{w}  (\forall x) t:A$. By Lemma \ref{lem:truth agree on the free variables M-model QLP}, since $y$ does not occur free in $t$ or $A$, we have $\M \Vdash_v  (\forall x) t:A$. Therefore, for every $r\in D$, $\M \Vdash_{v {x \choose r}} t:A$. It follows that, for every $r\in D$, $A\in\E(t^{v {x \choose r}} , v {x \choose r})$ and $\M \Vdash_{v {x \choose r}} A$. Thus, $(\forall x) A\in\E((t\forall x)^v , v)$ and $\M \Vdash_{v } (\forall x) A$. Hence, $\M\Vdash_v (t \forall x):(\forall x) A$.

Now we show that the rule qNec preserves validity. The proof of the validity preserving of other rules are straightforward. Suppose $A$ is valid in every strong M-model of $\QLP_\mathcal{F}$. Let $\M=(D,\I,\E,\V)$ be an strong M-model for $\QLP_\mathcal{F}$, and $v$ be an arbitrary valuation. We will show $\M\Vdash_v (\exists x) x:A$,  where $x$ does not occur free  $A$. By the hypothesis, $\M\Vdash_v A$. Since $\M$ is strong, $A\in\E(r,v)$, for some $r\in D$. Since $x$ does not occur free  $A$, the valuations $v$ and $v {x \choose r}$ agree on free variables of $A$, and hence $A\in\E(r,v {x \choose r})$. Thus, $\M\Vdash_{v {x \choose r}} x:A$, and so $\M\Vdash_v (\exists x) x:A$.
\qed
\end{proof}

Extending M-models to multi-agent quantified logic of proofs $\QLP^-_n(\mathcal{F})$, with primitive term specification $\mathcal{F}$, is straightforward. An M-model for $\QLP^-_n(\mathcal{F})$ is a quadruple $\M=(D,\I,\E,\V)$ such that $D$, $\I$, and $\V$ are defined as in Definition \ref{def:M-model QLP^-}, but now $\E$ is a mapping from the set of agents $\mathcal{A}$ to evidence functions, i.e. for each $i\in\mathcal{A}$, $\E(i)$ (or $\E_i$ for short) is an evidence function as defined in Definition \ref{def:M-model QLP^-} satisfying Application, Sum, Proof checker, and Primitive proof term conditions. Moreover, clause 9 in the definition of forcing relation should be replaced by
\begin{center}
$\M\Vdash_v t:_i A$ if{f} $A\in {\mathcal E}_i(t^v,v)$ and $\M\Vdash_v A$.
\end{center}

Now it is easy to show the following.
\begin{theorem}[Soundness $\QLP^-_n$]
Every formula provable in $\QLP^-_n(\mathcal{F})$ is valid in every M-model of $\QLP^-_n(\mathcal{F})$.
\end{theorem}

In the sequel, as in Section \ref{sec:The Surprise Test Paradaox in QLP}, we consider a class with only one student $s$ and use  $``:"$ instead of $``:_s"$. Thus, for simplicity, we reason in $\QLP^-$ instead of $\QLP^-_n$.

Our purpose is now to show Theorems \ref{thm:Knower Paradox in QLP^-}, \ref{thm:The Examiner Paradox in QLP^-}, \ref{thm:self one-day surprise without unless QLP^-}, and \ref{thm:self-reference tow-day Surprise Test Paradox in QLP^-} by constructing countermodels (using the soundness theorem). These theorems are restated here for convenience:

\begin{equation}\label{eq:The Knower Paradox in QLP^-}
 \QLP^-(\delta \leftrightarrow \neg (\exists x)x:\delta)_\emptyset\not\vdash\bot,
 \end{equation}
 \begin{equation}\label{eq:The Examiner Paradox in QLP^-}
\QLP^-(\delta \leftrightarrow [ (\exists x)x:\neg \delta\vee( E\wedge \neg (\exists x)x:(\delta\r E))])_\emptyset\not\vdash \bot,
\end{equation}
 \begin{equation}\label{eq:self one-day surprise without unless QLP^-}
  \QLP^-(\delta \leftrightarrow [  E\wedge \neg (\exists x)x:(\delta\r E)])_\emptyset\not\vdash \neg\delta.
  \end{equation}
  \begin{equation}\label{eq:self-reference tow-day Surprise Test Paradox in QLP^-}
\QLP^-(\delta \leftrightarrow [  E_1\wedge \neg (\exists x)x:(\delta\r E_1)]\veebar[  E_2\wedge \neg (\exists x)x:(\delta\wedge\neg E_1\r E_2)])_\emptyset \not\vdash \neg\delta
\end{equation}

Note that in all of the above theorems $\delta$'s are fixed point operators. Therefore we should actually extend the aforementioned semantics of $\QLP^-$ to ${\sf QLP^-(FP)}$, and give semantic interpretation for fixed point operators. But let us simply assume for now that fixed point operators are new propositional variables that are not in the language of $\QLP^-$. Recall that for $\QLP^-$-formula $F$,  $\QLP^-(F)$ denotes the extension of $\QLP^-$ by axiom $F$. Now it is easy to see the following soundness theorem for $\QLP^-(F)$.

\begin{theorem}\label{thm:soundness QLP^-(F)}
Given a fixed point axiom $F$, every formula provable in $\QLP^-(F)_\mathcal{F}$ is valid in every M-model of $\QLP^-_\mathcal{F}$ in which $F$ is valid.
\end{theorem}

In order to show (\ref{eq:The Knower Paradox in QLP^-})-(\ref{eq:self one-day surprise without unless QLP^-}), define M-model $\M=(D,\I,\E,\V)$ for $\QLP^-_\emptyset$ as follows:

\begin{enumerate}
\item let $D$  be an arbitrary non-empty set, and $\I$ an arbitrary interpretation on $D$,
\item let $\E(r,v)=\emptyset$ for all $r\in D$ and all valuations $v$, and
\item let $\V(\delta)=\V(E)=1$, the precise value of other propositional variables does not matter.
\end{enumerate}

 It is not difficult to show that $\M$ is a model of $\QLP^-_\emptyset$, and the following formulas are valid in $\M$
 \[\delta \leftrightarrow \neg (\exists x)x:\delta,\]
  \[\delta \leftrightarrow [ (\exists x)x:\neg \delta\vee( E\wedge \neg (\exists x)x:(\delta\r E))],\]
   \[\delta \leftrightarrow [  E\wedge \neg (\exists x)x:(\delta\r E)],\]
   \[\delta.\]
 Now, using Theorem \ref{thm:soundness QLP^-(F)} and the model $\M$,  (\ref{eq:The Knower Paradox in QLP^-})-(\ref{eq:self one-day surprise without unless QLP^-}) can be shown. In order to show (\ref{eq:self-reference tow-day Surprise Test Paradox in QLP^-}), consider the M-model $\M=(D,\I,\E,\V)$ defined as above, but in this case put $\V(\delta)=\V(E_1)=1$, and $\V(E_2)=0$. It is not difficult to show that $\M$ is a model of $\QLP^-_\emptyset$, and the following formulas  are valid in $\M$

 \[\delta \leftrightarrow [ E_1\wedge \neg (\exists x)x:(\delta\r E_1)]\veebar[ E_2\wedge \neg (\exists x)x:(\delta\wedge\neg E_1\r E_2)],\]
   \[\delta.\]
Again, using Theorem \ref{thm:soundness QLP^-(F)} and the model $\M$,  (\ref{eq:self-reference tow-day Surprise Test Paradox in QLP^-}) can be shown.


\end{document}